\documentclass[11pt,a4paper,oneside,onecolumn,fleqn]{article}
\usepackage{mathrsfs}
\usepackage{amsfonts}
\usepackage{latexsym}
\usepackage{amsmath,amsthm}
\usepackage{amssymb}
\usepackage{epsf}
\usepackage{graphicx}
\usepackage{indentfirst}
\usepackage{cite}
\usepackage{color}
\usepackage{diagbox}
\usepackage{booktabs}
\usepackage{multirow}

\usepackage{caption}

\renewcommand{\arraystretch}{1}

\theoremstyle{plain}
\newtheorem{thm}{\bf Theorem}[section]
\newtheorem{con}[thm]{\bf Construction}
\newtheorem{cor}[thm]{\bf Corollary}

\newtheorem{lem}[thm]{\bf Lemma}

\theoremstyle{definition}

\newtheorem{exam}{\bf Example}

\theoremstyle{remark}

\title{\bf On the existence of mixed orthogonal arrays of strength two with six factors}
\author{ Xiuling Shan$^{1,2,3}$, Guangzhou Chen$^{4*}$, Xiangrong Sun$^{1}$ \\
\small1. {\it School of Mathematical Sciences, Hebei
Normal University, Shijiazhuang, 050024, P.R.China}\\
\small2. {\it Hebei Research Center of the Basic Discipline Pure Mathematics, Shijiazhuang 050024, P.R.China}\\
\small3. {\it
Hebei Key Laboratory of Computational Mathematics and Applications, Shijiazhuang 050024, P.R.China}\\
\small4. {\it School of Mathematics and Statistics, Henan Normal University, Xinxiang, 453007, P.R.China}
\\}
\date{}
\begin{document}
\maketitle
\begin{center}
\begin{minipage}{16cm}
{
\begin{center}
{\bf Abstract}
\end{center}

\vspace{0.2cm} \ \ \ \
Orthogonal arrays are foundational objects in combinatorial design theory and experimental design theory, with broad applications in computer science, coding theory, and cryptography. In this paper, we investigate the existence for mixed orthogonal arrays of strength two with six factors, and propose several new constructions based on incomplete orthogonal Latin squares, incomplete transversal designs, and $\Sigma$-resolvable transversal designs. As a result, we show that the necessary conditions for the existence of such mixed orthogonal arrays are also sufficient, except for a small number of known exceptions and nine unsettled cases.


}
\end{minipage}
\end{center}

\vskip 0.5cm

{ \begingroup\makeatletter  \let\@makefnmark\relax  \footnotetext{$^*$Corresponding author: chenguangzhou0808@163.com}  \makeatother\endgroup}

\section{Introduction}

Orthogonal arrays have been extensively studied since their introduction
by Rao \cite{Rao0}. Numerous construction methods and theoretical results are presented in the monographs \cite{Hedayat,Dey} and the handbook \cite{Handbook}. Mixed (or asymmetric) orthogonal arrays, also introduced by Rao \cite{Rao}, have received considerable attention due to their important role in experimental design.
In particular, they provide universally optimal fractions for asymmetric factorial designs, see Cheng \cite{Cheng} and Mukerjee \cite{MuR}. Beyond their theoretical significance, mixed orthogonal arrays have found widespread applications in industrial experiments, particularly in quality improvement, as well as in various other experimental settings. With the rapid advancement of science and technology, investigating the existence of mixed orthogonal arrays with prescribed parameters has become an increasingly important problem.

A \emph{mixed} (or \emph{asymmetric}) \emph{orthogonal array}, denoted by MOA$(N; s_1s_2\cdots s_k, t)$, is an $N\times k$ array, in which the $i$-th column takes symbols from the set $S_i=\{0,1,\ldots,s_i-1\}$, for $1\leq i\leq k$, such that in any $N\times t$ subarray, every possible $t$-tuple occurs equally often as a row. In an MOA$(N;s_1s_2 \cdots s_k, t)$, the rows are called \emph{runs}, and thus $N$ denotes the number of runs, or equivalently, the \emph{size} of the array. Each symbol set $S_i(1\leq i\leq k )$ corresponds to a \emph{factor}, so $k$ is the number of factors. The integer $s_i$ is the \emph{level} of the $i$-th factor, and the sequence $s_1s_2\cdots s_k$ is referred to as the \emph{type} of the array. The parameter $t$ is called the \emph{strength}.

Note that the levels $s_1,s_2,\ldots,s_k$ need not be distinct. For convenience, repeated levels are usually grouped using exponential notation, for example, $2^22^33^2$ is written as $2^53^2$. It is straightforward to see that an MOA$(N;1s_2\cdots s_k,t)$ exists if and only if an MOA$(N;s_2\cdots s_k,t)$ exists.
Hence, unless otherwise stated, we assume throughout this paper that every MOA under consideration has no factor with only one level.

Clearly, the run size $N$ of an MOA$(N; s_1 s_2 \cdots s_k, t)$ must be
divisible by the least common multiple of all products $\prod_{i\in I}s_i$, where $I\subseteq \{1,2,\ldots,k\}$ and $|I|=t$. An MOA$(N; s_1 s_2
\cdots s_k, t)$ is called \emph{trivial} if every possible $k$-tuple
occurs equally often as a row.

We recall the following well-known bounds on the size of an orthogonal array (see \cite{Dey,Hedayat}).
\begin{lem}[\!\!\cite{Dey,Hedayat}]\label{bound}
Suppose an MOA$(N; s_1s_2\cdots s_k, t)$ exists, where $1\leq t\leq k$.

\begin{enumerate}
\item[(i)] If $t$ is even, then $N \geq
1+\sum\limits_{j=1}^{\frac{t}{2}}\sum\limits_{|I|=j}\prod\limits_{i\in I}(s_i-1)$.

\item[(ii)] If $t$ is odd, then a lower bound for $N$ is obtained by applying
the bound in (i) to the derived design MOA$(\frac{N}{s_k};
s_1s_2\cdots s_{k-1}, t-1)$, where $s_k=\max\{s_i: 1\leq i\leq k\}$.
\end{enumerate}
\end{lem}

A variety of construction techniques for mixed orthogonal arrays have been extensively developed in the literature. These include the difference matrix method \cite{Bose,Hedayat1,Ji,WJ0,WJ}, the compression replacement method \cite{Hedayat,Suen1}, constructions based on finite field \cite{Bose,CS2,ZhangT1}, the expansion replacement method \cite{Jiang,ZhangT2}, a grouping-based construction method \cite{Suen1}, orthogonal decomposition of projection matrix \cite{Zhang2,Zhang3}, several elementary construction techniques \cite{Br,Chen,Pang} and arithmetic constructions \cite{Nguyen}. For a comprehensive overview of known results on mixed orthogonal arrays, we refer the reader to \cite{Handbook,Hedayat,Wangzhi}.

The greatest common divisor of integers $s_1,s_2,\ldots,s_k$ is
denoted by $\gcd\{s_1,s_2,\ldots,s_k\}$, and their least common multiple is denoted by $\operatorname{lcm} \{s_1,s_2,\ldots,s_k\}$. The notation \(b\mid a\) means that the integer $a$ is divisible by $b$.

The existence of mixed orthogonal arrays of strength two
with two and three factors was given as an exercise
in \cite{Hedayat}.
\begin{lem}[\!\!\cite{Hedayat}]\label{k=23}
Let $s_1,s_2,s_3$ be positive integers. Then an
MOA$(N;s_1s_2,2)$ exists if and only if $s_1s_2 \mid N$, and an
MOA$(N;s_1s_2s_3,2)$ exists if and only if $M \mid N$, where
$M=\operatorname{lcm}\{s_1s_2,s_1s_3,s_2s_3\}$.
\end{lem}

The existence problem for mixed orthogonal arrays of strength two with
$4,5,6,\ldots$ factors was posed by Hedayat, Sloane, and Stufken \cite{Hedayat}. Subsequently, Chen, Ji, and Lei \cite{Chen} resolved the existence of strength two with four and five factors, with one possible exception, and further established the existence of strength $t\geq 3$ with $t+1$ factors. A summary of these results is given below.

\begin{lem}[\!\!\cite{Chen}]\label{k=4}
Let $a, b, c, d$ be integers greater than $1$. There is an
MOA$(N;abcd,2)$ of type $abcd\not\in \mathcal{G}=\{2^4,6^4\}\cup \{2^3d^1:
 d\equiv 2\pmod 4\}$
if and only if $M\mid N$, while there is an MOA$(N;abcd,2)$ of type $abcd\in \mathcal{G}$
if and only if $M\mid N$ and $N>M$, where $M=\operatorname{lcm}\{\alpha\beta:
\{\alpha,\beta\}\subset \{a,b,c,d\}\}$.
\end{lem}

\begin{lem}[\!\!\cite{Chen}]\label{k=5}
Let $a, b, c, d, e$ be integers greater than $1$. There is an
MOA$(N;abcde,2)$ of type $abcde\not\in {\cal
L}_5=\{2^5,3^5,6^42^1,6^43^1,6^5,10^5\}\cup \{2^4g^1:
 g\equiv 2\pmod 4\}$
if and only if $M\mid N$, while there is an MOA$(N;abcde,2)$ of type
$abcde\in {\cal L}_5\setminus \{10^5\}$ if and only if $M\mid N$ and
$N>M$, and there is an MOA$(100u;10^5,2)$ for $u>1$, where
$M=\operatorname{lcm}\{\alpha\beta: \{\alpha,\beta\}\subset \{a,b,c,d,e\}\}$.
\end{lem}

\begin{lem}[\!\!\cite{Chen}]\label{k=t+1}
There is an MOA$(M;a_1a_2\cdots a_{t+1},t)$, where
$M=\operatorname{lcm}\{\frac{a_1a_2\cdots a_{t+1}}{a_i}: 1\leq i\leq t+1$\}.
\end{lem}

\begin{thm}
Let $M=\operatorname{lcm}\{\frac{a_1a_2\cdots a_{t+1}}{a_i}: 1\leq i\leq t+1$\}. Then
there is an MOA$(N;a_1a_2\cdots a_{t+1},t)$ if and only if $M\mid N$.
\end{thm}
\begin{proof}
Lemma \ref{k=t+1} guarantees the existence of an MOA$(M;a_1a_2\cdots a_{t+1},t)$, which we denote by $A$. Since $M\mid N$, repeating $A$ for $\frac{N}{M}$ copies yields an MOA$(N;a_1a_2\cdots a_{t+1},t)$. The desired conclusion then follows directly from the necessary condition.
\end{proof}

Recently, Zhu et al. \cite{Zhu-Pang-2026} investigated the existence of mixed orthogonal arrays of strength two with seven factors and obtained the following results. Let $X=\{2,3,4,5,6,10,14,15,20,22,26,\\30,34,38,46\}$.

\begin{lem}[\!\!\cite{Abel-1,Handbook,Zhu-Pang-2026}]\label{k=7}
There is an MOA$(s^2;s^7,2)$ for any $s\not\in X$.
\end{lem}

\begin{cor}
There is an MOA$(N;s^7,2)$ for any $s\not\in X$, where
$s^2\mid N$.
\end{cor}
\begin{proof}
Lemma \ref{k=7} guarantees the existence of an MOA$(s^2;s^7,2)$, which we denote by $A$. Since $s^2\mid N$, repeating $A$ for $\frac{N}{s^2}$ copies yields an MOA$(N;s^7,2)$. The desired conclusion then follows directly from the necessary condition.
\end{proof}

\begin{lem}[\!\!\cite{Zhu-Pang-2026}]
Let $s_1,s_2$ be two distinct integers greater than one.
For any $s_1\notin X$, an MOA$(N;s_1^6s_2^1,2)$ exists if and only if  $\frac{s_1^2s_2}{r}\mid N$, where $r=\gcd\{s_1,s_2\}$.
\end{lem}

For $s_1\in X$, Zhu et al. \cite{Zhu-Pang-2026} presented several infinite families of an MOA$(\frac{s_1^2s_2}{r};s_1^6s_2^1,2)$, where $r=\gcd\{s_1,s_2\}$. For further  results on MOA of strength two with seven factors, we refer the reader to \cite{Zhu-Pang-2026}.

In this paper, we focus on the existence of mixed orthogonal arrays of strength two with six factors. By employing incomplete orthogonal Latin squares, incomplete transversal designs, and $\Sigma$-resolvable transversal designs, we develop several novel constructions and obtain the following result.

\begin{thm}\label{mainR}
Let $a, b, c, d, e, f$ be integers greater than $1$,
\begin{center}
$M=\operatorname{lcm}\{\alpha\beta: \{\alpha,\beta\}\subset \{a,b,c,d,e,f\}\}$,
\end{center}
\begin{center}
$\mathcal{E}=\{3^6, 4^6, 6^6, 2^53^1, 2^54^1, 4^52^1, 6^52^1, 6^53^1, 6^42^2, 6^43^2,
6^43^12^1\}\cup\{2^5(4k+2)^1: k\geq0\},$
\end{center}
\begin{center}
$\mathcal{P}=\{(2s)^6, (2s)^5s^1, (2s)^52^1, (2s)^4s^2: s\in\{5,11\}\}\cup\{10^42^15^1\}$.
\end{center}
There is an MOA$(N;abcdef,2)$ for $abcdef\not\in \mathcal{E}\cup \mathcal{P}$
if and only if $M\mid N$, while there is an MOA$(N;abcdef,2)$ for
$abcdef\in \mathcal{E}$ if and only if $M\mid N$ and
$N>M$, and there is an MOA$(uM;abcdef,2)$ for $abcdef\in \mathcal{P}$, where $u>1$.
\end{thm}

The rest of this paper is organized as follows. In Section $2$, we recall several general constructions for mixed orthogonal arrays, and review some auxiliary designs and known results, together with some new constructions of mixed orthogonal arrays derived from these designs. Section 3 is devoted to the proof of our main result. In the final section, we summarize the main results of this paper and discuss some directions for future research.

\section{Preliminaries}

In this section, we present several simple general constructions for mixed orthogonal arrays, as well as new constructions via incomplete orthogonal Latin squares, incomplete transversal designs, and $\Sigma$-resolvable transversal designs.

\subsection{General constructions}

In this subsection, we present several simple yet highly useful general constructions for mixed orthogonal arrays. We refer the reader to \cite{Br,Chen,Hedayat} and the references therein.

\begin{lem}[\!\!\cite{Br,Chen}](Taking multiples)\label{jb}
If there is an MOA$(N;a_1a_2\cdots a_k,t)$, then there is an
MOA$(nN; a_1a_2\cdots (na_k), t)$ for any positive
integer $n$.
\end{lem}

\begin{lem}[\!\!\cite{Br,Chen,Hedayat}](Juxtaposition construction)\label{jf}
If both an MOA$(N';a_1a_2\cdots a_k,t)$ and an
MOA$(N'';b_1a_2\cdots a_k,t)$ exist with $\frac{N'}{a_1}=\frac{N''}{b_1}$, then there is an MOA$(N'+N'';(a_1+b_1)a_2\cdots a_k,t)$.
\end{lem}

\begin{lem}[\!\!\cite{Br,Chen}](Split construction)\label{split}
If there is an MOA$(N; ba_3\cdots a_k,t)$ and $b=a_1a_2$, then there
is an MOA$(N; a_1a_2\cdots a_k,t)$.
\end{lem}

\begin{lem}[\!\!\cite{Chen}]\label{-1}
Suppose that $k-1\geq t$, $b=a_1a_2$ and $\gcd\{a_1,a_2\}=1$. If
there is an MOA$(N$; $b$$a_3\cdots a_k$, $t)$, where
$N=\operatorname{lcm}\{\prod\limits_{i=1}^t
x_i:\{x_1,x_2,\ldots,x_t\}\subseteq\{b,a_3,\ldots,a_k\}\}$, then there
is an MOA$(N_0$; $a_1$$a_2\cdots$ $a_k$, $t)$, where
\begin{center}
$N_0=\operatorname{lcm}\{\prod\limits_{i=1}^ty_i:\{y_1,y_2,\ldots,y_t\}\subseteq\{a_1,a_2,a_3,\ldots,a_k\}\}=N$.
\end{center}
\end{lem}

\begin{lem}[\!\!\cite{Chen}](Compression construction)\label{jj}
If there is an MOA$(N;a_1a_2\cdots a_k,t)$ and $d_i \mid  a_i$ for
$1\leq i\leq k$, then there is an MOA$(N;d_1d_2\cdots d_k,t)$.
\end{lem}

\begin{lem}[\!\!\cite{Chen}](Product construction)\label{jb1}
If both an MOA$(N_1;a_1a_2\cdots a_k,t)$ and an
MOA$(N_2;b_1\\b_2\cdots b_k,t)$ exist, then there is an
MOA$(N_1N_2;(a_1b_1)(a_2b_2)\cdots (a_kb_k),t)$.
\end{lem}

\begin{lem}[\!\!\cite{Chen}](Column-deletion construction)\label{quyinz}
If there is an MOA$(N;a_1a_2\cdots a_k,t)$, then there is an
MOA$(N;a_{i_1}a_{i_2}\cdots a_{i_s},t)$, where $s\leq k$ and
$\{i_1,i_2,\ldots, i_s\}\subseteq \{1,2,\ldots,k\}$.
\end{lem}

\begin{lem}[\!\!\cite{Chen}]\label{jf1}
If an MOA$(M_1;a_1a_2\cdots a_k,t)$ and an
MOA$(M_2;a_1a_2\cdots a_k,t)$ exist, then there is an
MOA$(xM_1+yM_2;a_1a_2\cdots a_k,t)$, where $x$ and $y$ are nonnegative
integers.
\end{lem}

The following lemma can be regarded as a generalization of the above, its proof is omitted.
\begin{lem}\label{jf2}
If an
MOA$(N_i;a_{1}a_{2}\cdots $ $a_{k},t)$ exists for $1 \leq i\leq m$, then there is an MOA$(\sum \limits_{i=1}^{m}x_{i}N_{i}$; $a_{1}a_{2}\cdots a_{k},t)$ for any positive integers $x_{1}, x_{2},\ldots, x_{m}$.
\end{lem}

Since every integer $u\geq 2$ can be written as $u=2x+3y$ for some nonnegative integers $x$ and $y$, the following corollary follows from Lemma \ref{jf1}.

\begin{cor}\label{jf11}
If an MOA$(2M;a_1a_2\cdots a_k,t)$ and an
MOA$(3M;a_1a_2\cdots a_k,t)$ exist, then an
MOA$(uM;a_1a_2\cdots a_k,t)$ exists for any integer $u\geq 2$.
\end{cor}

A well-known method for computing the least common multiple of a set of integers is stated in the following lemma.

\begin{lem}[\!\!\cite{Chen}]\label{common}
Let $m$ be a common multiple of integers
$a_1, a_2, \ldots, a_k$. Then
\begin{center}
$\operatorname{lcm}\{a_1,a_2,\ldots,a_{k}\}=\frac{m}{\gcd\{\frac{m}{a_1},\frac{m}{a_2},\ldots,\frac{m}{a_k}\}}$.
\end{center}
\end{lem}

For a comprehensive list of mixed orthogonal arrays, the readers may consult Table 12.7 of the monograph \cite{Hedayat} or the webpage
\begin{center}
http://support.sas.com/techsup/technote/ts723b.pdf.
\end{center}
In what follows, we present only those MOAs that are used in the subsequent sections of this paper.

\begin{lem}[\!\!\cite{Hedayat,CS,CS1,SC,WJ,Handbook,Wangzhi}]\label{a=2346}
There is an MOA$(ma^2;a^6,2)$ for $a\in \{2,3,4,6,10\}$ and $m=2,3$.
\end{lem}

\begin{lem}[\!\cite{Hedayat,CS,CS1,WJ,Handbook,Julian1,Wangzhi}]\label{ua=2346}
(1) Both an MOA$(8u;2^54^1,2)$ and an MOA$(8u;2^5(4u)^1,2)$ exist for $u=2,3$.

(2) There is an MOA$(ua^2;a^5(ua)^1,2)$ for $a\in \{3,4\}$ and $u=2,3$.

(3) There is an MOA$(2a^2;a^5(2a)^1,2)$ for $a\in \{6,10,22\}$.

(4) There is an MOA$((2u)^2;2^4(2u)^2,2)$ for $u=2,3$.

(5) An MOA$(16;4^42^2, 2)$, an MOA$(20;2^55^1,2)$, an MOA$(24;2^44^16^1,2)$ and an MOA$(100;\\10^42^2,2)$ exist.
\end{lem}

\begin{lem}[\!\!\cite{Hedayat,WJ0,Wangzhi}]\label{36100}
(1) There is an MOA$(36;6^3a^1b^1c^1,2)$ for $abc\in \{2^3,3^3,2^23^1,3^22^1\}$.

(2) There is an MOA$(100;10^3a^1b^1c^1,2)$ for $abc\in \{2^3,5^3,2^25^1,5^22^1\}$.

(3) There is an MOA$(484;22^311^3,2)$.
\end{lem}

\begin{lem}[\!\!\cite{Chen,Schoen}]\label{notexist}
There do not exist an MOA$(8;2^54^1,2)$, an MOA$(12;2^53^1,2)$, or an MOA$(36;6^ts^1,2)$ for $t\geq 4$ and $s \mid  6$.
\end{lem}

\begin{lem}\label{not16}
There does not exist an MOA$(16;4^52^1,2)$.
\end{lem}
\begin{proof}
The conclusion follows immediately from Lemma \ref{bound}.
\end{proof}

\subsection{Construction via incomplete Latin squares}

To present further constructions of MOAs, we recall the notion of Latin squares.
A \emph{Latin square} of order $n$ is an $n\times n$ array in which
each cell contains a single symbol from an $n$-set $S$, such that
each symbol occurs exactly once in each row and exactly once in each
column. Let $L_1$ and $L_2$ be Latin squares of order $n$ with
entries from sets $X$ and $Y$, respectively. They are called \emph{orthogonal
Latin squares} if for every pair $(x,y) \in X\times Y$, there is a unique cell $(i, j)$ satisfying $L_1(i, j) = x$ and $L_2(i, j) = y$.
A set of Latin squares of order $n$, $L_1,L_2,\ldots,L_m$ is called
\emph{mutually orthogonal} if $L_i$ and $L_j$ are orthogonal for any $1\leq i<j\leq m$. Such a set is abbreviated as an $m$-MOLS$(n)$.
An $\alpha$-\emph{transversal} in a Latin square of order $n$ is a set
of $n\alpha$ cells containing \(\alpha\) cells from each row and \(\alpha\) cells from each column, such that every symbol appears exactly \(\alpha\) times.

An \emph{incomplete Latin square}, denoted by ILS$(n; b_1, b_2,\ldots, b_k)$, is an $n\times n$ array $A$ with entries from an $n$-set $V$, together with subsets $V_i\subseteq V$ for $1\leq i\leq k$, where $\mid V_i\mid  = b_i$ and $V_i\cap V_j =\emptyset$ for $1\leq i\neq j \leq k$. Moreover, the following conditions hold:

1. each cell of $A$ is either empty or contains an element of $V$;

2. the subarrays indexed by $V_i\times V_i$ are empty (these subarrays are called \emph{holes}); and

3. the elements in row or column $b$ are exactly those of $V\setminus V_i$ if $b\in V_i$, and of $V$ otherwise.

\vskip 6pt
To illustrate the definition of an incomplete Latin square, we present the following example.
\begin{exam}
The array $A$ listed below is an ILS$(7; 1,1,1,2,2)$.
\end{exam}
\begin{center}
 {\renewcommand\arraystretch{0.8}
\setlength{\arraycolsep}{4pt}
\footnotesize
\hspace{15pt}
$A=\begin{array}{| c| c| c| c| c| c| c| }\hline	
&3& 4& 5 &6& 1& 2\\\hline
2 &&5 &6 &0 &4 &3\\\hline
1& 6&& 0& 5 &3& 4\\\hline
6& 5& 1&& &2& 0\\\hline
5 &2& 6&&& 0& 1\\\hline
3 &4 &0& 1& 2&&\\\hline
4& 0& 3& 2& 1&&	\\\hline
\end{array}$\ }
\end{center}

Two incomplete Latin squares ILS$(n;b_{1},b_{2},\ldots,b_{k})$ are \emph{orthogonal} if, when superimposed, every ordered pair belonging to
$(V\times V)\setminus \cup_{i=1}^{k}(V_{i}\times V_{i })$ appears exactly once. A pair of such squares is denoted by IMOLS$(n;b_{1},b_{2},\ldots,b_{k})$.
Similarly, $r$-IMOLS$(n;b_{1},b_{2},\ldots,b_{k})$ denotes a set of $r$ ILS$(n;b_{1},b_{2},\ldots,b_{k})s$ that are pairwise orthogonal.

\begin{exam}\label{ILS(6)}
Let \begin{center}
 {\renewcommand\arraystretch{0.8}
\setlength{\arraycolsep}{4pt}
\footnotesize
\hspace{15pt}
$A_1=\begin{array}{| c| c| c| c| c| c| c| }\hline	
4& 5& 2& 3& 0& 1\\\hline
1& 0& 5& 4& 2& 3\\\hline
5&4& 0& 1& 3& 2\\\hline
3& 2& 4& 5& 1& 0\\\hline
0& 3& 1& 2&&\\\hline
2& 1& 3& 0&&\\\hline
\end{array}$\ ,\ \ \
$A_2=\begin{array}{| c| c| c| c| c| c| }\hline	
0& 1& 4& 5& 2& 3\\\hline
5& 4& 0& 1& 3& 2\\\hline
3& 2& 5& 4& 0& 1\\\hline
4& 5& 3& 2& 1& 0\\\hline
1& 3& 2& 0&&\\\hline
2& 0& 1& 3&&\\\hline
\end{array}$}\ .
\end{center}
It is clear that the arrays $A_1$ and $A_2$ are two ILS$(6; 2)$s, and
it is straightforward to verify that they are orthogonal.
\end{exam}

We summarize several known results on MOLSs and IMOLSs in the following lemma.

\begin{lem}[\!\!\cite{Handbook,Abel-1,Todorov}]\label{MOLS}
(1) There is a $2$-MOLS$(n)$ for $n\not\in\{2,6\}$.

(2) There is a $3$-MOLS$(n)$ for $n\not\in\{2,3,6,10\}$.

(3) There is a $4$-MOLS$(n)$ for $n\not\in\{2,3,4,6,10,22\}$.

(4) Both a $4$-IMOLS$(22;3)$ and a $5$-IMOLS$(30; 5)$ exist.
\end{lem}

The following construction is based on the classic hole-filling idea. For related discussions, the reader may also refer to \cite{Pang}.

\begin{con}\label{daidong}
If both a $k$-IMOLS$(n; n_{1})$ and a $k_{1}$-MOLS$(n_{1})$ exist, then there is an
MOA$(n^2; n^{2+k_{1}}(\frac{n}{n'})^{k-k_{1}}, 2)$, where \(n'\) is an integer satisfying $n'\mid n$, $n'\geq n_{1}$, and $k\geq k_{1}$.
\end{con}
\begin{proof}
Let $\mathcal{X}=\mathbb{Z}_{n}=\{0, 1, 2,\ldots, n-1 \}$ and
$\mathcal{H}=\{n-n_1, n-n_1+1, n-n_1+2,\ldots, n-1 \}$.
By assumption, let  $A_{1}, A_{2}, \ldots, A_{k}$ be a $k$-IMOLS$(n; n_{1})$
over $(\mathcal{X},\mathcal{H})$, and let $B_{1}, B_{2}, \ldots, B_{k_{1}}$
be a $k_{1}$-MOLS$(n_{1})$ over $\mathcal{H}$. Filling the hole of $A_{1}, A_{2}, \ldots, A_{k_1}$ with $B_{1}, B_{2}, \ldots, B_{k_{1}}$, respectively, yields
a $k_{1}$-MOLS$(n)$ over $\mathcal{X}$, denote these squares by
$C_{1}, C_{2},\ldots, C_{k_1}$, where $C_p=(c_{i,j}^{(p)})$, $0\leq i,j\leq n-1$,
$1\leq p\leq k_1$. Write $n=un'$ since $n' \mid  n$, and define a mapping
\begin{center}
$f:\ \ \mathbb{Z}_{n} \longrightarrow \mathbb{Z}_{u}$,
$s\longmapsto \lfloor \frac{s}{n'} \rfloor$,
\end{center}
where $\lfloor x\rfloor$ denotes the greatest integer not exceeding $x$. Let $A_l=(a_{i,j}^{(l)})$ for $k_1+1\leq l\leq k$, $0\leq i,j\leq n-1$. We then define $D_l=(d_{i,j}^{(l)})$, where

$$d_{i,j}^{(l)}=\left\{
\begin{array}{lll}
\frac{n}{n'}-1, \ \ \ \  \emph{if} \ \ \emph{the}\ \emph{cell}\ (i,j)\ \emph{of}\ A_l \emph{ is} \  \emph{empty}, \\
f(a_{i,j}^{(l)}), \ \ \ \  \emph{otherwise}.
\end{array}
\right. $$
We now construct an array $E$ as follows.

$$E=\left(\begin{array}{cccccccc}
\textbf{0}&[n]&C_{0}^{(1)}&\cdots&C_{0}^{(k_1)}&D_{0}^{(k_1+1)}&\cdots&D_{0}^{(k)}\\
\textbf{1}&[n]&C_{1}^{(1)}&\cdots&C_{1}^{(k_1)}&D_{1}^{(k_1+1)}&\cdots&D_{1}^{(k)}\\
\vdots&\vdots&\vdots&\vdots&\vdots&\vdots&\vdots&\vdots\\
\textbf{n-1}&[n]&C_{n-1}^{(1)}&\cdots&C_{n-1}^{(k_1)}&D_{n-1}^{(k_1+1)}&\cdots&D_{n-1}^{(k)}\\
\end{array}\right),$$
where
$$\textbf{m}=\left(\begin{array}{c}
m\\
m\\
\vdots\\
m
\end{array}\right)_{n\times 1}, \ \
[n]=\left(\begin{array}{c}
0\\
1\\
\vdots\\
n-1
\end{array}\right),\ \
C_m^{(p)}=\left(\begin{array}{c}
c_{m,0}^{(p)}\\
c_{m,1}^{(p)}\\
\vdots\\
c_{m,n-1}^{(p)}
\end{array}\right),\ \ \
D_m^{(l)}=\left(\begin{array}{c}
d_{m,0}^{(l)}\\
d_{m,1}^{(l)}\\
\vdots\\
d_{m,n-1}^{(l)}
\end{array}\right),$$
$m\in\{0,1,2,\ldots,n-1\}$, $1\leq p\leq k_1$, $k_1+1\leq l\leq k$.
It is straightforward to verify that $E$ is an
MOA$(n^2; n^{2+k_{1}}(\frac{n}{n'})^{k-k_{1}}, 2)$. This completes the proof.
\end{proof}

The following example illustrates the above construction.

\begin{exam}
An MOA$(36;6^33^1,2)$ and an MOA$(36;6^32^1,2)$ exist, and were first given in \cite{Finney}.
\end{exam}
\begin{proof}
Let the arrays $A_1=(a_{i,j}^{(1)})$ and $A_2=(a_{i,j}^{(2)})$ be as in Example \ref{ILS(6)}. We obtain a Latin square of order $6$, denoted by $C_1=(c_{i,j})$,  by filling the hole of $A_1$ with a Latin square $B_1$ of order $2$.
\begin{center}
{\renewcommand\arraystretch{0.9}
\setlength{\arraycolsep}{4pt}
\footnotesize
\hspace{15pt}
$C_1=\begin{array}{| c| c| c| c| c| c| c| }\hline	
4& 5& 2& 3& 0& 1\\\hline
1& 0& 5& 4& 2& 3\\\hline
5&4& 0& 1& 3& 2\\\hline
3& 2& 4& 5& 1& 0\\\hline
0& 3& 1& 2&{\color{red}4}&{\color{red}5}\\\hline
2& 1& 3& 0&{\color{red}5}&{\color{red}4}\\\hline
\end{array}$\ , ~~\ \
$B_1=\begin{array}{| c| c| }\hline
4 & 5 \\\hline
5 & 4 \\\hline
\end{array}$\ .}
\end{center}
Define two mappings
\begin{center}
$f:\ \mathbb{Z}_{6} \longrightarrow \mathbb{Z}_{3}$,
$s\longmapsto \lfloor \frac{s}{2} \rfloor$,~ \ and $g:\ \mathbb{Z}_{6} \longrightarrow \mathbb{Z}_{2}$,
$s\longmapsto \lfloor \frac{s}{3} \rfloor$,
\end{center}
where $\lfloor\cdot\rfloor$ denotes the floor function.
We then define $D_1=(d_{i,j}^{(1)})$ and $D_2=(d_{i,j}^{(2)})$, where

$$d_{i,j}^{(1)}=\left\{
\begin{array}{lll}
2, \ \ \ \ \ \ \ \ \emph{if} \ \ \emph{the}\ \emph{cell}\ (i,j)\ \emph{of}\ A_2\ \emph{is} \ \emph{empty},\\
f(a_{i,j}^{(2)}), \ \   \emph{otherwise}.
\end{array}
\right.$$\ \ \
$$d_{i,j}^{(2)}=\left\{
\begin{array}{lll}
1, \ \ \ \ \ \ \ \ \emph{if} \ \ \emph{the}\ \emph{cell}\ (i,j)\ \emph{of}\ A_2\ \emph{is} \ \emph{empty},\\
g(a_{i,j}^{(2)}), \ \   \emph{otherwise}.
\end{array}
\right. $$
\vskip 6pt
\begin{center}
 {\renewcommand\arraystretch{0.9}
\setlength{\arraycolsep}{4pt}
\footnotesize
\hspace{15pt}
$D_1=\begin{array}{| c| c| c| c| c| c| }\hline	
0& 0& 2& 2& 1&1\\\hline
2& 2& 0& 0& 1& 1\\\hline
1& 1& 2& 2& 0& 0\\\hline
2& 2& 1& 1& 0& 0\\\hline
0& 1& 1& 0&2&2\\\hline
1& 0& 0& 1&2&2\\\hline
\end{array}$\ ,~~\
$D_2=\begin{array}{| c| c| c| c| c| c| }\hline	
0& 0& 1& 1& 0& 1\\\hline
1& 1& 0& 0& 1& 0\\\hline
1& 0& 1& 1& 0& 0\\\hline
1& 1& 1& 0& 0& 0\\\hline
0& 1& 0& 0&1&1\\\hline
0& 0& 0& 1&1&1\\\hline
\end{array}$
}\ .
\end{center}
\noindent
It is straightforward to verify that the sets $\{(i, j, c_{i,j}, d_{i,j}^{(1)})\}$ and $\{(i, j,  c_{i,j}, d_{i,j}^{(2)})\}$,
$i,j\in \mathbb{Z}_6$, form an MOA$(36;6^33^1,2)$ and an MOA$(36;6^32^1,2)$, respectively.
\end{proof}

Taking $k_1=k$ in the above construction yields the following result.
\begin{lem}[\!\!\cite{Handbook}]
If a $k$-IMOLS$(n; n_{1})$ and a $k$-MOLS$(n_{1})$ exist, then a $k$-MOLS$(n)$ exists, which also yields an MOA$(n^2;n^{k+2},2)$.
\end{lem}

\begin{lem}\label{22-30}
An MOA$(22^{2};22^{4}2^{2},2)$, an MOA$(30^{2};30^{6}5^{1},2)$ and an MOA$(30^{2};30^{6}6^{1},2)$ exist.
\end{lem}
\begin{proof}
A $4$-IMOLS$(22;3)$ and a $2$-MOLS$(3)$ exist by Lemma \ref{MOLS}.
Applying Construction \ref{daidong} with $n=22$, $k=4$, $k_1=2$, $n_1=3$ and $n'=11$, we obtain an MOA$(22^{2};22^{4}2^{2},2)$.

A $5$-IMOLS$(30; 5)$ and a $4$-MOLS$(5)$ exist by Lemma \ref{MOLS}.
Applying Construction \ref{daidong} with $n=30$, $k=5$, $k_1=4$, $n_1=5$ and $n'=5,6$, we obtain an MOA$(30^{2};30^{6}6^{1},2)$ and an MOA$(30^{2};30^{6}5^{1},2)$.
\end{proof}

\subsection{Construction via transversal designs}

A \emph{transversal design} TD$_\lambda(k, n)$ is a triple
$(\mathcal{X},\mathcal{G},\mathcal{A})$, where $\mathcal{X}$ is a
$kn$-element set of points, $\mathcal{G}$ is a collection of $n$-element subsets of
$\mathcal{X}$ (called \emph{groups}) that partition $\mathcal{X}$, and
$\mathcal{A}$ is a collection of subsets of $\mathcal{X}$ (called
\emph{blocks}), each block intersects every group in exactly one point, and
any pair of points from distinct groups appears in exactly
$\lambda$ blocks of $\mathcal{A}$. For a TD$_\lambda(k, n)$, if
$\lambda$ is omitted, it is taken to be one. A TD is said to be
\emph{$\alpha$-resolvable} if its blocks can be partitioned into
classes (called $\alpha$-\emph{resolution classes}) such that each point of
the design lies in precisely $\alpha$ blocks in each class. We call
a $1$-resolvable TD$_\lambda(k, n)$ a resolvable
TD$_\lambda(k, n)$, denoted by RTD$_\lambda(k, n)$.

\begin{lem}[\!\!\cite{Handbook}]\label{equivalent}
The following existence statements are equivalent:

(1) a $(k-2)$-MOLS$(n)$,

(2) a TD$(k,n)$,

(3) an RTD$(k-1,n)$,

(4) an MOA$(n^2;n^k,2)$.
\end{lem}

It is also well-known that an $\alpha$-parallel class of TD$(k,n)$ is
equivalent to a common $\alpha$-transversal in the $k-2$ mutually
orthogonal Latin squares of order $n$. Consequently, an
$\alpha$-resolvable TD$(k,n)$ exists if and only if the $k-2$ mutually
orthogonal Latin squares of size $n$ admit $\frac{n}{\alpha}$ disjoint
$\alpha$-transversals. The relationship between $\alpha$-resolvable TDs and MOAs is established in the following lemma.
\begin{lem}[\!\!\cite{Chen}]\label{rtd}
An $\alpha$-resolvable TD$_\lambda(k, n)$ exists if and only if
an MOA$(\lambda n^2;n^k(\frac{\lambda n}{\alpha})^1,2)$ exists.
\end{lem}

\begin{lem}\label{65181}
There is an MOA$(108;6^518^1,2)$.
\end{lem}
\begin{proof}
By Lemma \ref{rtd}, it suffices to construct an RTD$_3(5,6)$, which is provided in Appendix I.
\end{proof}

More generally, given a list of positive integers $\Sigma = [\sigma_1,\sigma_2,\ldots,\sigma_s ]$, a TD$_\lambda(k, n)$
is said to be $\Sigma$-\emph{resolvable} if its collection of blocks $\mathcal{B}$ admits a partition into sub-collections $\mathcal{B}_1, \mathcal{B}_2,\ldots, \mathcal{B}_s$, where for each $i=1,2,\ldots,s$, there is a $\sigma_i\in \Sigma$ such that $\mathcal{B}_i$ forms a $\sigma_i$-parallel class (or a $\sigma_i$-resolution classes).

An \emph{incomplete transversal design}, TD$_\lambda(k, u)-$TD$_\lambda(k, v)$, is a TD$_\lambda(k, u)$ with a subdesign TD$_\lambda(k, v)$ removed. We adopt the notation ITD$_\lambda(k,u;v)$ for such an incomplete transversal design TD$_\lambda(k, u) -$TD$_\lambda(k, v)$. Such an ITD may exist even when the omitted TD$_\lambda(k, v)$
does not. However, if the omitted subdesign exists, it can be combined with the incomplete transversal design to produce a TD$_\lambda(k, u)$.

\begin{lem}[\!\!\cite{Handbook}]\label{ITD=IM}
The existence of an ITD$(k,u;v)$ is
equivalent to the existence of a $(k-2)$-IMOLS$(u;v)$.
\end{lem}

\begin{lem}[\!\!\cite{Ge}]\label{TianD}
If an ITD$_\lambda(k,u;v)$ and a TD$_\lambda(k,v)$ exist, then
a TD$_\lambda(k,u)$ exists. Furthermore, if an IRTD$_\lambda(k,u;v)$ and an RTD$_\lambda(k,v)$ exist, then an RTD$_\lambda(k,u)$ exists.
\end{lem}

\begin{lem}\label{22lam=3}
There is an MOA$(1452;22^566^1,2)$.
\end{lem}
\begin{proof}
By Lemma \ref{MOLS}, a $4$-IMOLS$(22;3)$ exists. Hence, an ITD$(6,22;3)$ exists by Lemma \ref{ITD=IM}, which further implies the existence of an IRTD$(5,22;3)$, and repeated it three times to obtain an IRTD$_3(5,22;3)$.
An RTD$_3(5,3)$ exists, as it is equivalent to an MOA$(27;3^59^1,2)$ from Lemma \ref{ua=2346}. Applying Lemma \ref{TianD}, we then obtain an RTD$_3(5,22)$, which yields an MOA$(1452;22^566^1,2)$.
\end{proof}

The proof of the following result is similar to that of Construction \ref{daidong} and is therefore omitted.

\begin{con}\label{daidongTD}
If both an ITD$_\lambda(k,n;n_1)$ and a TD$_\lambda(k_1,n_1)$ exist, then there is an MOA$(\lambda n^2; n^{k_{1}}(\frac{n}{n'})^{k-k_{1}}, 2)$, where $n'\mid n$, $n'\geq n_{1}$ and $k\geq k_{1}$.
\end{con}

The following is a special case of Construction 2 in \cite{Rees1}, due to Rees. A similar construction by the same author later appeared in \cite{Rees,Rees2} in the context of transversal designs.

\begin{con}[\!\!\cite{Rees1}]\label{rtd1}
If there exists a $[1,m]$-resolvable TD$_{\lambda_1}(k, n)$ and an RTD$_{\lambda_2}(k,m)$, then
there exists an RTD$_{\lambda_1\lambda_2}(k,mn)$.
\end{con}

\begin{lem}\label{15RTD}
There is a $[1,5]$-resolvable TD$_3(5,2)$.
\end{lem}
\begin{proof}
Let the point set be $\mathcal{X}=\mathbb{Z}_{10}$, the group set be
$\mathcal{G}=\{\{0,5\},\{1,6\},\{2,7\},\{3,8\},\{4,9\}\}$ and the block set be
$\mathcal{B}=\{b_1,b_2,\cdots,b_{12}\}$, where

\mbox{}\hspace{0.4in}
$b_1=\{0,1,2,3,4\}$, $b_2=\{5,6,7,8,9\}$, $b_3=\{0,1,2,8,9\}$,  $b_4=\{0,1,7,3,9\}$,

\mbox{}\hspace{0.4in}
$b_5=\{0,6,2,8,4\}$,   $b_6=\{0,6,7,3,9\}$,
$b_7=\{0,6,7,8,4\}$,  $b_8=\{5,1,2,8,9\}$,

\mbox{}\hspace{0.34in}
$b_9=\{5,1,7,3,4\}$,  $b_{10}=\{5,1,7,8,4\}$,
$b_{11}=\{5,6,2,3,9\}$,  $b_{12}=\{5,6,2,3,4\}$.

\noindent
It is easily checked that $(X,\mathcal{G},\mathcal{B})$ is a $[1,5]$-resolvable TD$_3(5,2)$, where
$\mathcal{P}_1=\{b_1,b_2\}$ is a $1$-parallel class and
 $\mathcal{P}_2=\{b_3,b_4, \ldots ,b_{12}\}$ is a $5$-parallel class.
\end{proof}

\begin{lem}\label{300}
There is an MOA$(300;10^530^1,2)$.
\end{lem}
\begin{proof}
By Lemma \ref{15RTD}, there exists a $[1,5]$-resolvable TD$_3(5,2)$, and Lemma \ref{MOLS} guarantees the existence of an RTD$(5,5)$.
Applying Construction \ref{rtd1} yields an RTD$_3(5,10)$. Consequently, the desired
MOA$(300;10^530^1,2)$ is obtained via Lemma \ref{rtd}.
\end{proof}

\section{Proof of Theorem \ref{mainR}}

This section is devoted to the proof of Theorem \ref{mainR}.

\begin{lem}\label{lem3.1}
Let $a,b,c,d,e$ and $f$ be positive integers greater that $1$,
$M=\operatorname{lcm}\{\alpha\beta:\{\alpha,\beta\}\subseteq \{a,b,c,d,e,f\}\}$.
If there is $\{x,y\}\subseteq\{a,b,c,d,e,f\}$ such that $\gcd\{x,y\}=1$, then

(1) there exists an MOA$(M;abcdef,2)$ for any $abcdef\not\in\{6^{4}3^{1}2^{1},10^{4}5^{1}2^{1},2^{5}3^{1}\}$,
and

(2) there exists an MOA$(uM;abcdef,2)$ for any $abcdef\in\{6^{4}3^{1}2^{1},10^{4}5^{1}2^{1},2^{5}3^{1}\}$, where $u>1$.
\end{lem}
\begin{proof}
Without loss of generality, suppose that $x=a$ and $y=b$, and set $N=\operatorname{lcm}\{uv:\{u,v\}\subseteq\{ab,c,d,e,f\}\}$.
By Lemma \ref{k=5}, there exists an MOA$(N;(ab)^1cdef,2)$ except for $(ab)^1cdef\in \{6^42^1,6^43^1,6^5\}\cup \{(4k+2)^12^4: k\geq 1\}$
and except possibly for $(ab)^1cdef=10^5$. It is readily verified that
$M=N$ since $\gcd\{a,b\}=1$. By Lemma \ref{-1}, there exists an MOA$(M;abcdef,2)$ except for $abcdef\in \{6^33^12^2,6^33^22^1,6^43^12^1\}\cup \{(2k+1)^12^5: k\geq 1\}$ and except possibly for $abcdef=10^45^12^1$.

There does not exist an MOA$(36;6^{4}3^{1}2^{1},2)$, since there is no an MOA$(36;6^{4},2)$ by Lemma \ref{notexist}. However, both an MOA$(36;6^32^23^1,2)$ and an MOA$(36;6^33^22^1,2)$ exist by Lemma \ref{36100}. Furthermore, by Lemma \ref{k=5}, there exists an MOA$((2a)^2 u;(2a)^5,2)$ for any $u>1$ and $a\in\{3,5\}$.
Applying Lemma \ref{split} then yields an MOA$((2a)^2 u;(2a)^4a^12^1,2)$.

Next, we consider the existence of an MOA$(2(4k+2);2^5(2k+1)^1,2)$ for $k\geq1$.
For $k=1$, Lemma \ref{notexist} shows that an MOA$(12;2^53^1,2)$ does not exist.
On the other hand, Lemma \ref{ua=2346} gives an MOA$(24;2^44^16^1,2)$ and Lemma \ref{36100} gives an MOA$(36;6^32^23^1,2)$. Then applying Lemma \ref{jj} yields
an MOA$(24;2^53^1,2)$ and an MOA$(36;2^53^1,2)$. Consequently,
an MOA$(12u;2^53^1,2)$ exists any $u>1$ by Corollary \ref{jf11}.
For $k=2$, there exists an MOA$(20;2^55^1,2)$ by Lemma \ref{ua=2346}.
For $k\geq3$, Lemma \ref{a=2346} ensures the existence of an MOA$(8;2^6,2)$.
Applying Lemma \ref{jb} with $n=k-2$ then yields an MOA$(8(k-2);2^5(2(k-2))^1,2)$. Additionally, applying Lemma \ref{jf} with an MOA$(20;2^55^1,2)$ gives an MOA$(2(4k+2);2^5(2k+1)^1,2)$.
\end{proof}

In what follows, we only consider positive integers $a,b,c,d,e,f$ satisfying
$\gcd\{x,y\}>1$ for any subset $\{x,y\}\subseteq\{a,b,c,d,e,f\}$. We divide the analysis into the following six cases:\\
(1) $a=b=c=d=e=f$; \\
(2) $a=c=d=e=f\neq b$; \\
(3) $a=d=e=f$, $a\neq b$, $a\neq c$;\\
(4) $a=e=f$, $a\neq b$, $a\neq c$, $a\neq d$;\\
(5) $a=f$, $a\neq b$, $a\neq c$, $a\neq d$, $a\neq e$ and no three of
$b,c,d,e$ are all equal; and \\
(6) $a,b,c,d,e,f$ are pairwise distinct.

For convenience, we set $T=\{2,3,4,6,10,22\}$ throughout this section.

\subsection*{ \textbf{Case 1: \ $a=b=c=d=e=f$}}

\begin{lem}\label{lem3.2.1}
$(1)$ There is an MOA$(a^2;a^6,2)$ for any $a\not\in T$.

$(2)$ There is an MOA$(ua^2;a^6,2)$ for any $a\in T$, where $u>1$.
\end{lem}
\begin{proof}
$(1)$ The conclusion immediately holds by Lemma \ref{MOLS} and Lemma \ref{equivalent}.

$(2)$ For $a\in \{2,3,4,6\}$, Lemma \ref{a=2346} guarantees the existence of both an  MOA$(2a^2;a^6,2)$ and an MOA$(3a^2;a^6,2)$. Consequently, by Corollary \ref{jf11}, an
MOA$(ua^2;a^6,2)$ exists for every integer $u>1$. For $a\in\{10,22\}$, write $a=2s$ for $s\in\{5,11\}$. Note that an MOA$(s^2;s^6,2)$ and an MOA$(4u;2^6,2)$ exist. Applying Lemma \ref{jb1} then yields an MOA$((2s)^2u;(2s)^6,2)$, which is precisely an MOA$(ua^2;a^6,2)$.
\end{proof}

\subsection*{ \textbf{Case 2: $a=c=d=e=f\neq b$}}

\begin{lem}\label{lem3.2.3}
Let $a,b$ be integers greater than $1$ with $a\neq b$, $\gcd\{a,b\}>1$, and let $M= \operatorname{lcm}\{a^2,ab\}=\frac{a^2b}{\gcd\{a,b\}}$.

(1) There exists an MOA$(M;a^5b^1,2)$ for any $a\not\in T$.

(2) An MOA$(8;2^54^1)$ does not exist. If $b\equiv 0 \pmod 4$ and $b\neq 4$, then there exists an MOA$(2b;2^5b^1,2)$; if
$b\equiv 2 \pmod 4$, then there does not exist an MOA$(2b;2^5b^1,2)$,
while there exists an MOA$(2bu;2^5b^1,2)$ for any integer $u>1$.

(3) There exists an MOA$(3b;3^5b^1,2)$ for any integer $b$ satisfying $3\mid b$ and $b>3$.

(4) There exists an MOA$(M;a^5b^1,2)$ for $a\in\{4,6\}$ except for $a^5b^1\in \mathcal{A}_1=\{4^52^1,6^52^1,6^53^1\}$, while
there exists an MOA$(ua^2;a^5b^1,2)$ for $a^5b^1\in\mathcal{A}_1$, where $u>1$.

(5) There exists an MOA$(M;a^5b^1,2)$ for $a\in\{10,22\}$
except possibly for $a^5b^1\in\mathcal{A}_2=\{10^52^1,10^55^1,22^52^1,22^511^1\}$.
For $a^5b^1\in\mathcal{A}_2$, there exists an MOA$(ua^2;a^5b^1,2)$, where $u>1$.
\end{lem}
\begin{proof}
(1) Let $\gcd\{a,b\}=r$. Then $M=\frac{a^2b}{r}$. By Lemmas \ref{MOLS} and  \ref{equivalent}, there exists an MOA$(a^2;a^6,2)$ for any $a\not\in T$. Since $r\mid  a$, an MOA$(a^2;a^5r^1,2)$ exists by Lemma \ref{jj}. Applying Lemma \ref{jb} with $n=\frac{b}{r}$ we derive an MOA$(\frac{a^2b}{r};a^5b^1,2)$.

(2) For $a=2$, $b$ must be even, so $M=2b$. When $b\equiv0 \pmod{4}$, write $b=4k$ with $k\geq 1$. For $k=1$, no MOA$(8;2^54^1,2)$ exists by Lemma \ref{notexist}.
By Lemma \ref{ua=2346}, an MOA$(24;2^512^1,2)$ exists. Applying Lemma \ref{jj} yields an MOA$(24;2^54^1,2)$. Lemma \ref{ua=2346} gives an MOA$(16;2^54^1,2)$.
It then follows from Corollary \ref{jf11} that an MOA$(8u;2^54^1,2)$ exists for all integers $u>1$. For $k\geq 2$, it can be written as $k=2x+3y$, where $x$ and $y$ are nonnegative integers with $x+y>0$. By Lemma \ref{ua=2346}, an MOA$(16;2^58^1,2)$ and an MOA$(24;2^512^1,2)$ exist. Applying Lemma \ref{jb} with $n=x$ and $n=y$, respectively, we obtain an MOA$(16x;2^5(8x)^1,2)$ and MOA$(24y;2^5(12y)^1,2)$. Then
applying Lemma \ref{jf} yields an MOA$(16x+24y;2^5(8x+12y)^1,2)$, that is, an MOA$(2b;2^5b^1,2)$.

When $b\equiv 2~\pmod {4} $. Suppose an MOA$(2b;2^5b^1,2)$ exists. Then applying Lemma \ref{quyinz} yields an MOA$(2b;2^4b^1,2)$, which contradicts Lemma \ref{k=5}.
Therefore, no MOA$(2b;2^5b^1,2)$ exists. By Lemma \ref{a=2346}, an MOA$(8;2^6,2)$ and an MOA$(12;2^6,2)$ exist. Applying Lemma \ref{jb} with $n=\frac{b}{2}$ in each case, we obtain an MOA$(4b;2^5b^1,2)$ and an MOA$(6b;2^5b^1,2)$. It then follows from Corollary \ref{jf11} that an MOA$(2bu;2^5b^1,2)$ exists for all integers $u>1$.

(3) For $a=3$ with $b\neq a$ and $\gcd\{a,b\}=r>1$, so $r=3$ and $b=3k$ with $k\geq 2$. Consequently, $M=3b$. For every integer $k\geq 2$, it can be written as $k=2x+3y$, where $x$ and $y$ are nonnegative integers with $x+y>0$. By Lemma \ref{ua=2346}, an MOA$(18;3^56^1,2)$ and an MOA$(27;3^59^1,2)$ exist. By Lemma \ref{jb} with
$n=x$ and $n=y$, respectively, we obtain an MOA$(18x;3^5(6x)^1,2)$ and an MOA$(27y;3^5(9y)^1,2)$. Applying Lemma \ref{jf} then yields an MOA$(18x+27y;3^5(6x+9y)^1,2)$, that is, an MOA$(3b;3^5b^1,2)$.

(4) For $a=4$ with $b\neq a$ and $\gcd\{a,b\}=r>1$, then $b$ must be even.
When $b\equiv 0 \pmod{4}$, write $b=4(2x+3y)=8x+12y$, where $x$ and $y$ are nonnegative integers with $x+y>0$. Lemma \ref{ua=2346} gives an MOA$(32;4^58^1,2)$ and an MOA$(48;4^512^1,2)$. We obtain an MOA$(32x;4^5(8x)^1,2)$ and an MOA$(48y;4^5(12y)^1,2)$ by Lemma \ref{jb} with $n=x$ and $n=y$, respectively.
Then by Lemma \ref{jf}, there exists an MOA$(32x+48y;4^5(8x+12y)^1,2)$, that is, an MOA$(4b;4^5b^1,2)$.

When $b\equiv 2\pmod {4}$, write $b=4k+2$, where $k\geq 0$.
When $k=0$, that is, $b=2$. No MOA$(16;4^52^1,2)$ exists by Lemma \ref{not16},
while Lemma \ref{lem3.2.1} gives an MOA$(16u;4^6,2)$ for $u>1$, then applying
Lemma \ref{jj} yields an MOA$(16u;4^52^1,2)$. When $k=1$, then $b=6$.
There exists an MOA$(48;4^512^1,2)$ by Lemma \ref{ua=2346}, then there exists an MOA$(48;4^56^1,2)$ by Lemma \ref{jj}. When $k\geq2$, an MOA$(32;4^6,2)$ exists by Lemma \ref{lem3.2.1}. Applying Lemma \ref{jb} with $n=k-1$, we obtain an MOA$(32k-32;4^5(4k-4)^1,2)$. Combining this with the MOA$(48;4^56^1,2)$ via Lemma \ref{jf}, we further get an MOA$(32k+16;4^5(4k+2)^1,2)$. Therefore, an MOA$(8b;4^5b^1,2)$ exists for all integers $b\equiv 2\pmod {4}$ with $b\geq 6$.

For $a=6$ with $b\neq a$ and $\gcd\{a,b\}=r>1$. By assumption, we have $b\equiv 0,2,3,4 \pmod {6}$. When $b\equiv 0\pmod{6}$, write $b=6(2x+3y)=12x+18y$,
where $x$ and $y$ are nonnegative integers with $x+y>0$, then $M=6b$.
An MOA$(72;6^512^1,2)$ and an MOA$(108;6^518^1,2)$ exist by Lemmas \ref{ua=2346} and \ref{65181}, respectively. Applying Lemmas \ref{jb} and \ref{jf} yields an MOA$(72x+108y;6^5(12x+18y)^1,2)$, which is precisely an MOA$(6b;6^5b^1,2)$.

When $b\equiv 2\pmod{6}$, write $b=6k+2$, where $k\geq 0$, which gives $M=18b$.
For $k=0$, no MOA$(36;6^52^1,2)$ exists by Lemma \ref{notexist}. However,
by Lemma \ref{lem3.2.1}, an MOA$(36u;6^6,2)$ exists for $u>1$. Applying Lemma \ref{jj} then yields an MOA$(36u;6^52^1,2)$. For $k=1$, Lemma \ref{ua=2346} gives an MOA$(72;6^512^1,2)$, then there exists an MOA$(72;6^54^1,2)$ by Lemma \ref{jj}, and subsequently, by Lemma \ref{jb} with $n=2$, we obtain an MOA$(144;6^58^1,2)$.
For $k\geq 2$, Lemma \ref{lem3.2.1} provides an MOA$(108;6^6,2)$. Applying Lemma \ref{jb} with $n=k-1$ produces an MOA$(108(k-1);6^5(6k-6)^1,2)$. Combining this with the MOA$(144;6^58^1,2)$ via Lemma \ref{jf} we obtain an
MOA$(108k+36;6^5(6k+2)^1,2)$, that is, an MOA$(18b;6^5b^1,2)$.

When $b\equiv 3 \pmod {6}$, we write $b=6k+3$, where $k\geq 0$, which implies  $M=12b$. For $k=0$, no MOA$(36;6^53^1,2)$ exists by Lemma \ref{notexist}. However,
by Lemma \ref{lem3.2.1}, an MOA$(36u;6^6,2)$ exists for any integer $u>1$. Applying Lemma \ref{jj} yields an MOA$(36u;6^53^1,2)$. For $k=1$, Lemma \ref{65181} gives an MOA$(108;6^518^1,2)$, then by Lemma \ref{jj}, we obtain an MOA$(108;6^59^1,2)$.
For $k\geq 2$, Lemma \ref{lem3.2.1} provides an MOA$(72;6^6,2)$. Applying Lemma
\ref{jb} with $n=k-1$ gives an MOA$(72(k-1);6^5(6k-6)^1,2)$. Combining this with the
MOA$(108;6^59^1,2)$ via Lemma \ref{jf}, we obtain an MOA$(72k+36;6^5(6k+3)^1,2)$, that is, an MOA$(12b;6^5b^1,2)$.

When $b\equiv 4 \pmod {6}$, we write $b=6k+4$, where $k\geq 0$, then $M=18b$.
For $k=0$, Lemma \ref{ua=2346} gives an MOA$(72;6^512^1,2)$. Applying Lemma \ref{jj} produces an MOA$(72;6^54^1,2)$.
For $k\geq 1$, Lemma \ref{lem3.2.1} provides an MOA$(108;6^6,2)$. Applying Lemma \ref{jb} with $n=k$ yields an MOA$(108k;6^5(6k)^1,2)$. Combining this with the MOA$(72;6^54^1,2)$ via Lemma \ref{jf}, we obtain an MOA$(108k+72;6^5(6k+4)^1,2)$, that is, an MOA$(18b;6^5b^1,2)$.

(5) For $a=10$ with $b\neq a$. Since $\gcd\{a,b\}=r>1$, we have $b\equiv 0,2,4,5,6,8 \pmod{10}$. When $b\equiv 0 \pmod {10}$, we write $b=10(2x+3y)=20x+30y$,
where $x$ and $y$ are nonnegative integers with $x+y>0$. Clearly, $M=10b$.
An MOA$(200;10^520^1,2)$ and an MOA$(300;10^530^1,2)$ exist by Lemmas \ref{ua=2346} and \ref{300}, respectively. Applying Lemma \ref{jb} with $n=x$ and $n=y$,
respectively, yields an MOA$(200x;10^5(20x)^1,2)$ and an MOA$(300y;10^5(30y)^1,2)$.
Then by Lemma \ref{jf}, we obtain an MOA$(200x+300y;10^5(20x+30y)^1,2)$, which is precisely an MOA$(10b;10^5b^1,2)$.

When $b\equiv 2 \pmod {10}$, we write $b=10k+2$, where $k\geq 0$. We have $M=50b$. For $k=0$, Lemma \ref{lem3.2.1} gives an MOA$(100u;10^6,2)$ for any integer $u>1$, and then applying Lemma \ref{jj} yields an MOA$(100u;10^52^1,2)$. For $k=1$, Lemma \ref{ua=2346} gives an MOA$(200;10^520^1,2)$. Applying Lemma \ref{jj} yields
an MOA$(200;10^54^1,2)$, and then by Lemma \ref{jb} with $n=3$, we obtain an MOA$(600;10^512^1,2)$. For $k\geq 2$, Lemma \ref{lem3.2.1} provides an MOA$(500;10^6,2)$. Applying Lemma \ref{jb} with $n=k-1$ produces an MOA$(500(k-1);10^5(10k-10)^1,2)$. Then, by Lemma \ref{jf} combined with the MOA$(600;10^512^1,2)$, we obtain an MOA$(50b;10^5b^1,2)$.

When $b\equiv w \pmod {10}$ for $w\in \{4,6,8\}$, write $b=10k+w$, where $k\geq 0$. Then $M=50b$. For $w\in \{4,6,8\}$. Lemma \ref{ua=2346} gives an MOA$(200;10^520^1,2)$. Applying Lemma \ref{jj} yields an MOA$(200;10^54^1,2)$, and then by Lemma \ref{jb} with $n=2$, we obtain an MOA$(400;10^58^1,2)$.
Lemma \ref{300} gives an MOA$(300;10^530^1,2)$, then by Lemma \ref{jj},
an MOA$(300;10^56^1,2)$ exists. Therefore, there exists an MOA$(50w;10^5w^1,2)$ for $w\in \{4,6,8\}$. Lemma \ref{lem3.2.1} gives an MOA$(500;10^6,2)$, then by Lemma \ref{jb} with $n=k$, we obtain an MOA$(500k;10^5(10k)^1,2)$. Applying Lemma \ref{jf}
yields an MOA$(500k+50w;10^5(10k+w)^1,2)$, that is, an MOA$(50b;10^5b^1,2)$.

When $b\equiv 5 \pmod{10}$, we write $b=10k+5$, where $k\geq 0$, then $M=20b$.
For $k=0$, Lemma \ref{lem3.2.1} gives an MOA$(100u;10^6,2)$ for any integer $u>1$, then applying Lemma \ref{jj} produces an MOA$(100u;10^55^1,2)$. For $k=1$,
Lemma \ref{300} guarantees the existence of an MOA$(300;10^530^1,2)$, then we obtain an MOA$(300;10^515^1,2)$ by Lemma \ref{jj}. For $k\geq 2$, Lemma \ref{lem3.2.1} ensures that the existence of an MOA$(200;10^6,2)$, then by Lemma \ref{jb} with $n=k-1$, an MOA$(200(k-1);10^5(10k-10)^1,2)$ exists. By Lemma \ref{jf} with the MOA$(300;10^515^1,2)$, we obtain an MOA$(20(10k+5);10^5(10k+5)^1,2)$, that is, an MOA$(20b;10^5b^1,2)$.

For $a=22$ with $b\neq a$. Since $\gcd\{a,b\}=r>1$, $b\equiv 0,2,4,6,8,10,11,12,14,16,18,20 \pmod{22}$.
When $b\equiv 0 \pmod {22}$, it can be written as $b=22(2x+3y)=44x+66y$,
where $x$ and $y$ are nonnegative integers with $x+y>0$, then $M=22b$.
Both an MOA$(22\times44;22^544^1,2)$ and MOA$(22\times66;22^566^1,2)$ exist
by Lemma \ref{ua=2346} and Lemma \ref{22lam=3}, respectively. Then by Lemma \ref{jb}
with $n=x$ and $n=y$, respectively, we obtain an MOA$(22\times44x;22^5(44x)^1,2)$ and MOA$(22\times66y;22^5(66y)^1,2)$. Consequently, by Lemma \ref{jf}, there exists an MOA$(22\times(44x+66y);22^5(44x+66y)^1,2)$, that is, an MOA$(22b;22^5b^1,2)$.

When $b\equiv w \pmod {22},\ w\in\{2,4,6,8,10,12,14,16,18,20\}$, let $b=22k+w$, where $k\geq 0$, then $M=22\times11b$. For $w=2$ and $k=0$, Lemma \ref{lem3.2.1} gives an MOA$(22^2u;22^6,2)$ for $u>1$, applying Lemma \ref{jj} yields an MOA$(22^2u;22^52^1,2)$. For  $w=2$ and $k=1$, there exists an MOA$(22\times44;22^544^1,2)$ by Lemma \ref{ua=2346}, we obtain an MOA$(22\times44;22^54^1,2)$ by Lemma \ref{jj}. Then by Lemma \ref{jb} with $n=6$, an MOA$(22\times11\times24;22^524^1,2)$ exists. For $w=2$ and $k\geq2$, Lemma \ref{lem3.2.1} gives an MOA$(22^2\times11;22^6,2)$, applying Lemma \ref{jb} with $n=k-1$ produces an MOA$(22^2\times11(k-1);22^5(22k-22)^1,2)$. Then by Lemma \ref{jf} with the MOA$(22\times11\times24;22^524^1,2)$, we obtain an MOA$(22\times11(22k+2);22^5(22k+2)^1,2)$. Next, Starting with an MOA$(22\times44;22^54^1,2)$, by Lemma \ref{jb}, we obtain an MOA$(22\times11w;22^5w^1,2)$ for $w\in\{4,8,12,16,20\}$. Lemma \ref{22lam=3} gives an MOA$(22\times66;22^566^1,2)$, then applying Lemma \ref{jj} yields an MOA$(22\times66;22^56^1,2)$. Consequently, applying Lemma \ref{jf} with an MOA$(22\times11\times w_1;22^5w_1^1,2)$ for $w_1\in\{4,8,12\}$ produces an
MOA$(22\times11w;22^5w^1,2)$ for $w\in\{10,14,18\}$. Therefore, there exists an MOA$(22\times11w;22^5w^1,2)$ for $w\in \{4,6,8,10,12,14,16,18,20\}$.
Lemma \ref{lem3.2.1} gives an MOA$(22^2\times11;22^6,2)$, applying Lemma \ref{jb} with $n=k$ produces an MOA$(22^2\times11k;22^5(22k)^1,2)$ for any integer $k\geq1$. Applying Lemma \ref{jf} with an MOA$(22\times11w;22^5w^1,2)$, we obtain an MOA$(22\times11b;22^5b^1,2)$.

When $b\equiv 11 \pmod {22}$, we write $b=22k+11$, where $k\geq 0$, then $M=22\times2b$. For $k=0$, an MOA$(22^2u;22^6,2)$ exists for any $u>1$ by Lemma \ref{lem3.2.1}, then there exists an MOA$(22^2u;22^511^1,2)$ by Lemma \ref{jj}. For $k=1$, by Lemma \ref{22lam=3}, an MOA$(22\times66;22^566^1,2)$ exists, then an MOA$(22\times66;22^533^1,2)$ exists. For $k\geq 2$, Lemma \ref{lem3.2.1} gives an MOA$(22\times44;22^6,2)$. Applying Lemma \ref{jb} with $n=k-1$ yields an MOA$(22\times44(k-1);22^5(22k-22)^1,2)$. Thus, by Lemma \ref{jf} with an MOA$(22\times66;22^533^1,2)$, an MOA$(22\times2b;22^5b^1,2)$ exists.
This completes the proof.
\end{proof}

\section*{Case 3: \ \ $a=d=e=f$, $a\neq b$ and $a\neq c$}

In this case, we consider the existence of an MOA$(M;a^4bc,2)$, where
$a\neq b$, $a\neq c$ and
\begin{center}
$M=
\operatorname{lcm}\{a^2,ab,ac,bc\}=\operatorname{lcm}\{\operatorname{lcm}\{a^2,ab\},\ \operatorname{lcm}\{ac,bc\}\}=\operatorname{lcm}
\{a,b\}\cdot \operatorname{lcm} \{a,c\}$.
\end{center}

\begin{lem}\label{lem3.2.4}
Let $a,b,c$ be positive integers greater than $1$ with $a\neq b$ and $a\neq c$. Suppose that $\gcd\{x,y\}>1$ for any subset $\{x,y\}\subseteq\{a,b,c\}$. Then

(1) an MOA$(M;a^4b^1c^1,2)$ exists for $a\not\in T$,

(2) an MOA$(M;a^4b^1c^1,2)$ exists for $a\in \{2,3,4\}$,

(3) an MOA$(M;a^4b^1c^1,2)$ exists for $a\in \{6,10,22\}$ except
for $a^4b^1c^1\in \mathcal{A}_3=\{6^42^2,6^42^13^1,\\6^43^2\}$, and except possibly for $a^4b^1c^1\in\mathcal{A}_4=\{10^45^2,10^45^12^1,22^411^2\}$, while there exists an MOA$(uM;a^4b^1c^1,2)$ for $a^4b^1c^1\in \mathcal{A}_3$, and an MOA$(uM;a^4b^1c^1,2)$ for $a^4b^1c^1\in \mathcal{A}_4$, where $u>1$.
\end{lem}
\begin{proof}
Suppose that $\gcd\{a,b\}=d_1,\ \gcd\{a,c\}=d_2$. Then  $M=\operatorname{lcm}\{a,b\}\cdot \operatorname{lcm}\{a,c\}=\frac{a^2bc}{d_1d_2}$.

(1) For $a\not\in T$, an MOA$(a^2;a^6,2)$ exists by Lemma \ref{lem3.2.1}. Applying
Lemma \ref{jb}, we obtain an MOA$(\frac{a^2bc}{d_1d_2};a^4(\frac{ab}{d_1})^1(\frac{ac}{d_2})^1,2)$. Since $b\mid \frac{ab}{d_1}$ and $c\mid \frac{ac}{d_2}$, Lemma \ref{jj} yields an MOA$(\frac{a^2bc}{d_1d_2};a^4b^1c^1,2)$.

(2) For the case $a=2$ with $b\neq2$ and $c\neq2$, then both $b$ and $c$ are
even, and we have $\gcd\{a,b\}=\gcd\{a,c\}=2$, and $M=bc$.
When $b\equiv c\equiv0 \pmod {4}$, by Lemma \ref{ua=2346}, an MOA$(16;2^44^2,2)$ exists. Applying Lemma \ref{jb}, we obtain an MOA$(bc;2^4b^1c^1,2)$.

When $b\equiv0 \pmod{4}, c\equiv2\pmod{4}$ or $b\equiv2\pmod{4}, c\equiv0\pmod{4}$.
We assume without loss of generality that $b\equiv0 \pmod {4},\ c\equiv2 \pmod {4}$.
Write $b=4k_1$, $c=4k_2+6$ with integers $k_1\geq1$, $k_2\geq0$. When $k_1=1$ and $k_2=0$, an MOA$(24;2^44^16^1,2)$ exists by Lemma \ref{ua=2346}.
For all other $k_1$ and $k_2$, an MOA$(16;2^44^2,2)$ exists by Lemma \ref{ua=2346}. Applying Lemma \ref{jb} with $n=k_2$  yields an MOA$(16k_2;2^44^1(4k_2)^1,2)$. By Lemma \ref{jf} with an MOA$(24;2^44^16^1,2)$, we obtain an MOA$(24+16k_2;2^44^1(4k_2+6)^1,2)$, that is, an MOA$(4c;2^44^1c^1,2)$. Consequently, by Lemma \ref{jb} with $n=\frac{b}{4}$, there exists an MOA$(bc;2^4b^1c^1,2)$.

When $b\equiv c\equiv2 \pmod{4}$. Write $b=4x+6$, $c=4y+6$, where $x$ and $y$ are integers with $x,y\geq0$. For $x=y=0$, an MOA$(36;2^46^2,2)$ exists by Lemma \ref{ua=2346}.
Applying Lemma \ref{jb} with $n=x$ to an MOA$(24;2^44^16^1,2)$, we obtain an
MOA$(24x;2^4(4x)^16^1,2)$. Combining this with the MOA$(36;2^46^2,2)$ via Lemma \ref{jf}, we get an MOA$(24x+36;2^4(4x+6)^16^1,2)$. From the result above, an
MOA$((4x+6)(4y);2^4(4x+6)^1(4y)^1,2)$ exists. Then applying Lemma \ref{jf} yields an
MOA$((4x+6)(4y+6);2^4(4x+6)^1(4y+6)^1,2)$, that is, an MOA$(bc;2^4b^1c^1,2)$.

For the case $a=3$ with $b\neq3$ and $c\neq3$. Then $3\mid  b$, $3\mid  c$ and $M=bc$. An MOA$(3b;3^5b^1,2)$ exists by Lemma \ref{lem3.2.3}. Applying Lemma \ref{jb} with $n=\frac{c}{3}$, we obtain an MOA$(bc;3^4b^1c^1,2)$.

For the case $a=4$ with $b\neq4$, $c\neq4$. Then $b$ and $c$ are both even.
When $b\equiv 0 \pmod {4}$ or $c\equiv 0 \pmod {4}$. We assume without loss of generality that $b\equiv0 \pmod 4$. Lemma \ref{lem3.2.3} gives an MOA$(4b;4^5b^1,2)$.
If $c\equiv0 \pmod {4}$, applying Lemma \ref{jb} with $n=\frac{c}{4}$ produces an MOA$(bc;4^4b^1c^1,2)$. If $c\equiv2 \pmod {4}$, applying Lemma \ref{jj} to an MOA$(4b;4^5b^1,2)$, an MOA$(4b;4^4b^12^1,2)$ exists. Then by Lemma \ref{jb} with $n=\frac{c}{2}$, we obtain an MOA$(2bc;4^4b^1c^1,2)$.

If $b\equiv c\equiv2 \pmod 4$. Write $b=4x+2$, $c=4y+2$, where $x$ and
$y$ are nonnegative integers. It is easy to see that $M=4bc$.
For $x=y=0$, Lemma \ref{ua=2346} gives an MOA$(16;4^42^2,2)$.
For $x\geq1$, by Lemma \ref{lem3.2.3}, an MOA$(8b;4^5b^1,2)$ exists. Applying
Lemma \ref{jj} yields an MOA$(8b;4^4b^12^1,2)$. Furthermore, by Lemma \ref{jb} with $n=\frac{c}{2}$, there exists an MOA$(4bc;4^4b^1c^1,2)$.

(3) We now consider the cases $a\in\{6,10,22\}$.
For $a\in\{6,10,22\}$ with $a\neq b$ and $a\neq c$. If $a\mid b$ or $a\mid c$,
we assume without loss of generality that $a\mid  b$ and $\gcd\{a,c\}=s,\ s\in \{2,\frac{a}{2},a\}$, then we have $M=\frac{abc}{s}$.
Lemma \ref{lem3.2.3} gives an MOA$(ab;a^5b^1,2)$. Applying Lemma \ref{jj} yields an MOA$(ab;a^4b^1s^1,2)$. Consequently, by Lemma \ref{jb} with $n=\frac{c}{s}$, there exists an MOA$(\frac{abc}{s};a^4b^1c^1,2)$. It remains to consider the case where neither $b$ nor $c$ is divisible by $a$.

For $a=6$, suppose that $\gcd\{6,b\}=s_1$, $\gcd\{6,c\}=s_2$, $s_1,s_2\in\{2,3\}$, then we have $M=\frac{6b}{s_1}\cdot\frac{6c}{s_2}=\frac{36bc}{s_1s_2}$. When $a^4b^1c^1\in \mathcal{A}_3=\{6^42^2,6^42^13^1,6^43^2\}$, no MOA$(36;a^4b^1c^1,2)$ exists by Lemma \ref{notexist}. Lemma \ref{lem3.2.3} gives an MOA$(36u;a^5b^1,2)$ for $u>1$ and $b\in\{2,3\}$, applying Lemma \ref{jj} yields an MOA$(36u;a^4b^1c^1,2)$ for any $a^4b^1c^1\in \mathcal{A}_3$. When $a^4b^1c^1\not\in \mathcal{A}_3$, $b\in\{2, 3\}$ and $c\not\in\{2,3\}$. By Lemma \ref{lem3.2.3}, an MOA$(\frac{36c}{s_2};a^5c^1,2)$ exists, thus by Lemma \ref{jj}, there exists an MOA$(\frac{36c}{s_2};a^4s_1^1c^1,2)$. Applying Lemma \ref{jb} with $n=\frac{b}{s_1}$ yields an MOA$(\frac{36bc}{s_1s_2};6^4b^1c^1,2)$. The existence of an MOA$(\frac{36bc}{s_1s_2};6^4b^1c^1,2)$ for the case $a^4b^1c^1\not\in \mathcal{A}_3$, $b\not\in\{2, 3\}$ and $c\in\{2,3\}$ is the same as the foregoing assertion.
When $a^4b^1c^1\not\in \mathcal{A}_3$ and $b,c\not\in\{2,3\}$. Lemma \ref{lem3.2.3} gives an MOA$(\frac{36b}{s_1};6^5b^1,2)$, then by Lemma \ref{jj}, we have an MOA$(\frac{36b}{s_1};6^4b^1s_2^1,2)$. Applying Lemma \ref{jb} with $n=\frac{c}{s_2}$ produces an MOA$(\frac{36bc}{s_1s_2};6^4b^1c^1,2)$.

For $a\in\{10,22\}$. Suppose that $\gcd\{a,b\}=t_1$, $\gcd\{a,c\}=t_2$, then $t_1,t_2\in\{2,\frac{a}{2}\}$, and $M=\frac{a^2bc}{t_1t_2}$.
When $a^4b^1c^1\in\mathcal{A}_4=\{10^45^2,10^45^12^1,22^411^2\}$, Lemma \ref{lem3.2.3} gives an MOA$(a^2u;a^5(\frac{a}{2})^1,2)$ for $u>1$. Applying Lemma \ref{jj} yields an MOA$(a^2u;a^4b^1c^1,2)$ for $a^4b^1c^1\in\mathcal{A}_4$ and $u>1$.
When $t_1=t_2=2$, an MOA$(a^2;a^42^2,2)$ exists by Lemma \ref{ua=2346} or Lemma \ref{22-30}. Thus, by Lemma \ref{jb} with $n=\frac{b}{2}$ and $n=\frac{c}{2}$, respectively, there exists an MOA$(\frac{a^2bc}{4};a^4b^1c^1,2)$.
When $t_1=t_2=\frac{a}{2}$ and $a^4b^1c^1\not\in\mathcal{A}_4$, then at least one of $b$ and $c$ is not equal to $\frac{a}{2}$. We assume without loss of generality that $b\neq\frac{a}{2}$. Lemma \ref{lem3.2.3} gives an MOA$(2ab;a^5b^1,2)$, then by Lemma \ref{jj}, we have an MOA$(2ab;a^4b^1(\frac{a}{2})^1,2)$. Applying Lemma \ref{jb} with $n=\frac{2c}{a}$ produces an MOA$(4bc;a^4b^1c^1,2)$.
When $t_1=2$, $t_2=\frac{a}{2}$ and $a^4b^1c^1\neq10^42^15^1$. For $b=2$, then $c\neq5$.  Lemma \ref{lem3.2.3} gives an MOA$(2ac;a^5c^1,2)$, then by Lemma \ref{jj}, we have an MOA$(2ac;a^42^1 c^1,2)$. Applying Lemma \ref{jb} with $n=\frac{b}{2}$ produces an MOA$(abc;a^4b^1c^1,2)$. For $b\neq2$. Lemma \ref{lem3.2.3} gives an MOA$(\frac{a^2b}{2};a^5b^1,2)$, then by Lemma \ref{jj}, we have an MOA$(\frac{a^2b}{2};a^4(\frac{a}{2})^1b^1,2)$. Applying Lemma \ref{jb} with $n=\frac{2c}{a}$ produces an MOA$(abc;a^4b^1c^1,2)$.
The existence of an MOA$(abc;a^4b^1c^1,2)$ for the case $t_1=\frac{a}{2}$, $t_2=2$, and $a^4b^1c^1\neq10^45^12^1$ is the same as the foregoing assertion.
\end{proof}

\section*{Case 4: \ \ $a=e=f$, $a\neq b$, $a\neq c$ and $a\neq d$}

In this case, we investigate the existence of an MOA$(M;a^3bcd,2)$, where $a=e=f$,
$a\neq b, c,d$, and
$M=\operatorname{lcm}\{a^2,ab,ac,ad,bc,bd,cd\}$ throughout this case.

\begin{lem}\label{3.2.5}
Suppose that $a,b,c,d$ be positive integers greater than $1$ with $a\neq b, c,d$,  and $\gcd\{x,y\}>1$ for any $\{x,y\}\subseteq \{a,b,c,d\}$, then an MOA$(M;a^3bcd,2)$ exists.
\end{lem}
\begin{proof}
Without loss of generality, let
\begin{center}
$\gcd\{a,b\}=r_1$,\ $\gcd\{a,c\}=r_2$,\ $\gcd\{a,d\}=r_3$,\ and $\gcd\{\frac{b}{r_1},\frac{c}{r_2},\frac{d}{r_3}\}=r_4$.
\end{center}
Then, by Lemma 4.7 in \cite{Chen}, we obtain $M=\frac{a^2bcd}{r_1r_2r_3r_4}$.

We now distinguish four cases according to the value of $a$:
(1) $a\not\in T=\{2,3,4,6,10,22\}$, (2) $a\in\{2,3\}$,
(3) $a=4$, (4) $a\in\{6,10,22\}$.

(1) For $a\not\in T$. Lemma \ref{MOLS} gives an MOA$(a^2;a^6,2)$, then applying Lemma \ref{jj} produces an MOA$(a^2;a^3r_1^1r_2^1r_3^1,2)$. By Lemma \ref{k=23}, there exists an MOA$(\frac{bcd}{r_1r_2r_3r_4}; (\frac{b}{r_1})^1(\frac{c}{r_2})^1(\frac{d}{r_3})^1,2)$, which can also be written as an MOA$(\frac{bcd}{r_1r_2r_3r_4}; 1^11^11^1(\frac{b}{r_1})^1(\frac{c}{r_2})^1(\frac{d}{r_3})^1,2)$. Then, by Lemma \ref{jb1}, it follows that there exists an MOA$(M;a^3b^1c^1d^1,2)$.

(2) For $a\in\{2,3\}$, $a\neq b, c,d$, and $a\mid  b$, $a\mid  c$,  $a\mid  d$.
Let $\gcd\{b,c,d\}=s$. It is straightforward to verify that $a\mid  s$ and $M=\frac{bcd}{s}$. When $s=a$, an MOA$(cd;a^4c^1d^1,2)$ exists by Lemma \ref{lem3.2.4}. Applying Lemma \ref{jb} with $n=\frac{b}{a}$, we obtain an MOA$(\frac{bcd}{a};a^3b^1c^1d^1,2)$.
When $a=2$ and $s\in\{4,6,10,22\}$, an MOA$(s^2;a^3s^3,2)$ exists by Lemmas \ref{ua=2346}-\ref{36100}, \ref{22-30} and \ref{jj}. When $a=3$ and $s=6$, an MOA$(s^2;a^3s^3,2)$ exists by Lemma \ref{36100}. For the remaining values of $a$ and $s$, an MOA$(s^2;s^6,2)$ exists by Lemmas \ref{MOLS} and \ref{equivalent}, then applying Lemma \ref{jj} yields an MOA$(s^2;a^3s^3,2)$. By Lemma \ref{k=23}, there exists an MOA$(\frac{bcd}{s^3};1^11^11^1(\frac{b}{s})^1(\frac{c}{s})^1(\frac{d}{s})^1,2)$. Therefore, an MOA$(\frac{bcd}{s};a^3b^1c^1d^1,2)$ exists by Lemma \ref{jb1}.

(3) For $a=4$. By assumption, we have $b\neq 4$, $c\neq 4$ and $d\neq 4$, and $r_i\in\{2,4\}$, $i=1,2,3$. Suppose that $\gcd\{b,c,d\}=s$. When $r_1=r_2=r_3=4$, then $ 4\mid  s$ and $M=\frac{bcd}{s}$. For $s=4$, by Lemma \ref{lem3.2.4}, an MOA$(cd;s^4c^1d^1,2)$ exists. Applying Lemma \ref{jb} with $n=\frac{b}{s}$ yields an MOA$(\frac{bcd}{s};s^3b^1c^1d^1,2)$.
For $ 4\mid  s$ and $s>4$, an MOA$(s^2;4^3s^3,2)$ exists by Lemma \ref{lem3.2.4} and Lemma \ref{jj}, and an
MOA$(\frac{bcd}{s^3};1^11^11^1(\frac{b}{s})^1(\frac{c}{s})^1(\frac{d}{s})^1,2)$ exists by Lemma \ref{k=23}. Then by Lemma \ref{jb1}, there exists an MOA$(M;a^3b^1c^1d^1,2)$. When $r_1=r_2=4$, $r_3=2$, it is straightforward to see that $r_4=\gcd\{\frac{b}{4},\frac{c}{4},\frac{d}{2}\}=\frac{s}{2}$ and $M=\frac{bcd}{s}$.
Since $\gcd\{4,d\}=r_3=2$, one can readily verify that $\operatorname{lcm}\{bc,b(2d),c(2d)\}=\frac{bcd}{s}$ by Lemma \ref{common}.
From the above, an MOA$(\frac{bcd}{s};4^3b^1c^1(2d)^1,2)$ exists, since $\gcd\{4,2d\}=4$. Applying Lemma \ref{jj}, we obtain an MOA$(M;4^3b^1c^1d^1,2)$.
When $\{r_1,r_2,r_3\}\in\{\{4,2,2\},\{2,2,2\}\}$, an MOA$(16;4^3r_1^1r_2^1r_3^1,2)$ exists by Lemmas \ref{lem3.2.4}
and \ref{jj}, and an MOA$(\frac{bcd}{r_1r_2r_3r_4}; 1^11^11^1(\frac{b}{r_1})^1(\frac{c}{r_2})^1(\frac{d}{r_3})^1,2)$ exists by Lemma \ref{k=23}. Then there exists an MOA$(M;a^3b^1c^1d^1,2)$ by Lemma \ref{jb1}.

(4) In what follows, we consider the remaining cases with $a\in\{6,10,22\}$. By our hypothesis, we have $b\neq a$, $c\neq a$, $d\neq a$ and $r_i\in\{2,\frac{a}{2},a\}$, $i=1,2,3$. Let $\gcd\{b,c,d\}=s$.  When $r_1=r_2=r_3=a$, then we have $ a\mid  s$ and  $M=\frac{bcd}{s}$. For $s=a$, an MOA$(cd;a^4c^1d^1,2)$ exists by Lemma \ref{lem3.2.4}. Applying Lemma \ref{jb} with $n=\frac{b}{a}$ yields an MOA$(\frac{bcd}{a};a^3b^1c^1d^1,2)$. For $ a\mid s$ and $s>a$, by Lemma \ref{lem3.2.1}, an MOA$(s^2;s^6,2)$ exists.
Applying Lemma \ref{jj} produces an MOA$(s^2;a^3s^3,2)$. By Lemma \ref{k=23},
an MOA$(\frac{bcd}{s^3};1^11^11^1(\frac{b}{s})^1(\frac{c}{s})^1(\frac{d}{s})^1,2)$ exists. Then there exists an MOA$(\frac{bcd}{s};a^3b^1c^1d^1,2)$ by Lemma \ref{jb1}.

When $r_1=a$, $\{r_2,r_3\}\in \{\{2,2\},\{2,\frac{a}{2}\},\{\frac{a}{2},\frac{a}{2}\},\{2,a\},\{\frac{a}{2},a\}\}$. Then we have $\gcd\{\frac{a}{r_2},c\}=1$, $\gcd\{\frac{a}{r_3},d\}=1$, $\frac{a}{r_2}\mid  b$ and $\frac{a}{r_3}\mid  b$. It follows that $\operatorname{lcm}\{b\cdot(\frac{a}{r_2}c),b\cdot(\frac{a}{r_3}d),(\frac{a}{r_2}c)\cdot
(\frac{a}{r_3}d)\}=\frac{a^2bcd}{r_1r_2r_3r_4}$ by Lemma \ref{common}.
From the above case $r_1=r_2=r_3=a$, since
$\gcd\{a,\frac{a}{r_2}c\}=a$ and $\gcd\{a,\frac{a}{r_3}d\}=a$, an MOA$(\frac{a^2bcd}{r_1r_2r_3r_4};a^3b^1(\frac{a}{r_2}c)^1(\frac{a}{r_3}d)^1,2)$ exists. It follows that an MOA$(\frac{a^2bcd}{r_1r_2r_3r_4};a^3b^1c^1d^1,2)$ exists by Lemma \ref{jj}.

When $\{r_1,r_2,r_3\}\in \{\{2,2,2\},\{2,2,\frac{a}{2}\},\{2,\frac{a}{2},\frac{a}{2}\},\{\frac{a}{2},\frac{a}{2},\frac{a}{2}\}\}$, an MOA$(a^2;a^3r_1^1r_2^1r_3^1,2)$ exists by Lemmas \ref{ua=2346}, \ref{36100} or Lemmas \ref{lem3.2.4} and \ref{jj}, and an MOA$(\frac{bcd}{r_1r_2r_3r_4}; 1^11^11^1(\frac{b}{r_1})^1(\frac{c}{r_2})^1(\frac{d}{r_3})^1,2)$ exists
by Lemma \ref{k=23}. Then by Lemma \ref{jb1}, an MOA$(M;a^3b^1c^1d^1,2)$ exists.
This completes the proof.
\end{proof}

\section*{Case 5:\ \ $a=f$, $a\neq b$, $a\neq c$, $a\neq d$, $a\neq e$ and no three of $b,c,d,e$ are all equal}

\begin{lem}\label{bbb}
Let $a,b,c,d,e,f$ be positive integers greater than $1$, suppose that
$\gcd\{a,f\}=r_1$, $\gcd\{r_1,b\}=r_2$, $\gcd\{r_1,c\}=r_3$,
$\gcd\{r_1,d\}=r_4$ and $\gcd\{r_1,e\}=r_5$. Let $M_6=\operatorname{lcm}\{xy:\{x,y\}\subseteq\{a,b,c,d,e,f\}\}$ and
$M'=\operatorname{lcm}\{\alpha\beta:\{\alpha,\beta\}\subseteq\{\frac{b}{r_2},
\frac{c}{r_3},\frac{d}{r_4},\frac{e}{r_5},\frac{a}{r_1},\frac{f}{r_1}\}\}$,
then $M_6=r_1^2M'$.
\end{lem}
\begin{proof}
By Lemma \ref{common}, we have $M_6=\frac{abcdef}{P}$, where
\begin{center}
$P=\gcd\{wxyz:\{w,x,y,z\}\subseteq\{a,b,c,d,e,f\}\}$.
\end{center}
Similarly,
$M'=\frac{\frac{afbcde}{r_1r_1r_2r_3r_4r_5}}{P'}$, where
\begin{center}
$P'=\gcd\{\alpha \beta \gamma \delta:\{\alpha,\beta,\gamma,\delta\}\subset
\{\frac{b}{r_2},\frac{c}{r_3},\frac{d}{r_4},\frac{e}{r_5},\frac{a}{r_1},\frac{f}{r_1}\}\}$.
\end{center}
To prove that $M_6=r_1^2M'$, it suffices to show that $P=r_2r_3r_4r_5P'$.
The proof of this identity is provided in Appendix II.
\end{proof}

In the sequel, we consider the existence of an MOA$(M;a^2bcde,2)$, where
$a\neq b$, $a\neq c$, $a\neq d$, $a\neq e$ and no three of
$b,c,d,e$ are all equal, let
$M=\operatorname{lcm}\{a^2,ab,ac,ad,ae,bc,bd,be,cd,ce,de\}$.

Setting $f=a$ in Lemma \ref{bbb} gives the following lemma.
\begin{lem}\label{aaa}
Let $a,b,c,d,e$ be positive integers greater than $1$, $\gcd\{a,b\}=r_2$, $\gcd\{a,c\}=r_3$, $\gcd\{a,d\}=r_4$ and $\gcd\{a,e\}=r_5$. If
$M'=\operatorname{lcm}\{\frac{bc}{r_2r_3},\frac{bd}{r_2r_4},\frac{be}{r_2r_5},\frac{cd}{r_3r_4},\frac{ce}{r_3r_5},\frac{de}{r_4r_5}\}$,
then $M=a^2M'$.
\end{lem}

\begin{lem}\label{ccc}
Let $a,b,c,d,e$ be positive integers greater than $1$, $a\not\in T=\{2,3,4,6,10,22\}$, $a\neq b,\ a\neq c,\ a\neq d,\ a\neq e$
and no three of $b,c,d,e$ are all equal. If $\gcd\{x,y\}>1$ for any $\{x,y\}\subseteq\{a,b,c,d,e\}$, then there exists an MOA$(M;a^2bcde,2)$.
\end{lem}
\begin{proof}
Suppose that $\gcd\{a,b\}=r_2$, $\gcd\{a,c\}=r_3$, $\gcd\{a,d\}=r_4$ and $\gcd\{a,e\}=r_5$. By assumption, $a\not\in T$, Lemma \ref{lem3.2.1} gives an MOA$(a^2;a^6,2)$, and then applying Lemma \ref{jj} yields an MOA$(a^2;a^2r_2r_3r_4r_5,2)$. By Lemma \ref{k=4}, there exists an MOA$(M';1^11^1\frac{b}{r_2}\frac{c}{r_3}\frac{d}{r_4}\frac{e}{r_5},2)$, where
$\frac{b}{r_2}\frac{c}{r_3}\frac{d}{r_4}\frac{e}{r_5}\not\in\mathcal{H}=\{6^4\}\cup\{2^3(4k+2)^1:k\geq0\}$ and $M'$ is defined as in Lemma \ref{aaa}. Consequently, applying Lemmas \ref{jb1} and \ref{aaa} yields an MOA$(M;a^2bcde,2)$.

It remains to consider the cases where $\frac{b}{r_2}\frac{c}{r_3}\frac{d}{r_4}\frac{e}{r_5}\in\mathcal{H}$.

For  the case $\frac{b}{r_2}\frac{c}{r_3}\frac{d}{r_4}\frac{e}{r_5}=s^4, \ s\in\{2,6\}$. we immediately have $b=s r_2$, $c=s r_3$, $d=s r_4$, $e=s r_5$.
Using the given $\gcd$ conditions $\gcd\{a,b\}=\gcd\{a,s r_2\}=r_2$, $\gcd\{a,c\}=\gcd\{a,s r_3\}=r_3$,
$\gcd\{a,d\}=\gcd\{a,sr_4\}=r_4$, $\gcd\{a,e\}=\gcd\{a,s r_5\}=r_5$,
we deduce that $r_i\mid a$ but $s r_i\nmid a$ for each $i=2,3,4,5$, with $r_i>1$.
Since $\frac{b}{r_2}=\frac{c}{r_3}=\frac{d}{r_4}=\frac{e}{r_5}=s$, it follows that
no three of $r_2,r_3,r_4,r_5$ are all equal; otherwise, the same equality would hold among $b,c,d,e$, contradicting the assumption. Hence, we obtain
$a\neq 3,5,11$ and $s a\not\in T$. By Lemma \ref{lem3.2.1}, there exists an MOA$((sa)^2;(sa)^6,2)$. Applying Lemma \ref{jj} then yields an MOA$((s a)^2;a^2(sr_2)^1(sr_3)^1(sr_4)^1(sr_5)^1,2)$. Finally, by Lemma \ref{aaa} with
$M=a^2M'=(sa)^2$, this is precisely an MOA$(M;a^2bcde,2)$ as required.

For the case $\frac{b}{r_2}\frac{c}{r_3}\frac{d}{r_4}\frac{e}{r_5}=2^3(4k+2)^1$, where $k>0$.
We may assume, without loss of generality, that $b=2r_2$, $c=2r_3$, $d=2r_4$, $e=(4k+2) r_5$. The given $\gcd$ conditions then become $\gcd\{a,e\}=\gcd\{a,(4k+2)r_5\}=r_5$, $\gcd\{a,b\}=\gcd\{a,2r_2\}=r_2$, $\gcd\{a,c\}=\gcd\{a,2r_3\}=r_3$, $\gcd\{a,d\}=\gcd\{a,2r_4\}=r_4$. Consequently,
we have $r_i\mid a$ but $s r_i\nmid a$ for each$i=2,3,4,5$, with $r_i>1$.
Since $b,c,d$ are not all equal, it follows that $r_2,r_3,r_4$ are not all equal either. Hence $a$ has at least two distinct factors greater than $1$, which implies that $2 a\not\in T$. By Lemma \ref{lem3.2.1}, there exists an MOA$((2a)^2;(2a)^6,2)$. Applying Lemma \ref{jj} then yields an MOA$((2a)^2;a^2(2r_2)^1(2r_3)^1(2r_4)^1(2r_5)^1,2)$. Now setting $n=2k+1$ in Lemma \ref{jb}, we obtain an
MOA$((2a)^2(2k+1);a^2(2r_2)^1(2r_3)^1(2r_4)^1(4k+2)r_5)^1,2)$, which, by Lemma \ref{aaa} with $M=a^2M'=a^2(2(4k+2))=(2a)^2(2k+1)$, is precisely an
MOA$(M;a^2bcde,2)$.
\end{proof}

\begin{lem}\label{fff}
Let $a,b,c,d$ and $e$ be positive integers greater than $1$, $a\in T=\{2,3,4,6,10,22\}$, $ a\neq b,\ a\neq c,\ a\neq d,\ a\neq e$
and no three of $b,c,d,e$ are all equal. If $\gcd\{x,y\}>1$ for any $\{x,y\}\subseteq\{a,b,c,d,e\}$ and there is a subset $\{x,y\}\subseteq\{a,b,c,d,e\}$ such that $\gcd\{x,y\}\not\in T$, then there exists an MOA$(M;a^2bcde,2)$.
\end{lem}
\begin{proof}
Without loss of generality, set $\gcd\{b,c\}=r_1\not\in T$, $\gcd\{r_1,a\}=s_1$, $\gcd\{r_1,d\}=s_2$, $\gcd\{r_1,e\}=s_3$. Since $\gcd\{\frac{b}{r_1},\frac{c}{r_1}\}=1$, it follows from Lemmas \ref{k=4}-\ref{k=5} and Lemma \ref{lem3.1} that there exists an MOA$(M_1;\frac{b}{r_1}\frac{c}{r_1}(\frac{a}{s_1})^2\frac{d}{s_2}\frac{e}{s_3},2)$
for $\frac{b}{r_1}\frac{c}{r_1}(\frac{a}{s_1})^2\frac{d}{s_2}\frac{e}{s_3}\not\in\mathcal{T}
=\{s^51^1:s\in\{2,3,6,10\}\}\cup\{2^41^2,6^41^2,6^43^11^1,6^42^11^1,6^43^12^1,2^53^1,10^45^12^1\}$
$\cup\{2^4(4k+2)^11^1,2^3(4k+2)^11^2:k\geq1\}$, and $M_1=\operatorname{lcm}\{xy:\{x,y\}\subseteq\{\frac{b}{r_1},\frac{c}{r_1},\frac{a}{s_1},\frac{a}{s_1},\frac{d}{s_2},\frac{e}{s_3}\}\}$.
By Lemmas \ref{lem3.2.1} and \ref{jj}, there exists an MOA$(r_1^2;r_1^2s_1^2s_2^1s_3^1,2)$. Finally, applying Lemma \ref{jb1} together with Lemma \ref{bbb}, and setting $M=r_1^2M_{1}$, we obtain the desired MOA$(M;a^2bcde,2)$.

We next consider the case where $\frac{b}{r_1}\frac{c}{r_1}(\frac{a}{s_1})^2\frac{d}{s_2}\frac{e}{s_3}\in\mathcal{T}$.

Suppose first that $\frac{b}{r_1}\frac{c}{r_1}(\frac{a}{s_1})^2\frac{d}{s_2}\frac{e}{s_3}=s^51^1$,
$s\in\{2,3,6,10\}$, then $M=r_1^2M_1=r_1^2s^2$. Without loss of generality, we mey assume $\frac{b}{r_1}=1, \frac{c}{r_1}=\frac{a}{s_1}=\frac{d}{s_2}=\frac{e}{s_3}=s$
since $\gcd\{\frac{b}{r_1},\frac{c}{r_1}\}=1$. Hence $r_1=b$. Now, because $\gcd\{b,x\}>1$ for each $x\in\{a,c,d,e\}$, and no three of $b,c,d,e$ are all equal, it follows that at least two of $s_1,s_2,s_3$ are distinct. Consequently, $r_1$ has at least two distinct factors greater than $1$, which implies that $sr_1\not\in T$. By Lemma \ref{lem3.2.1}, there exists an MOA$((sr_1)^2;(sr_1)^6,2)$. Then, we obtain an MOA$((sr_1)^2;(sr_1)^1r_1^1(ss_1)^2(ss_2)^1(ss_3)^1,2)$ by Lemma \ref{jj}, which is precisely the desired MOA$(M;a^2bcde,2)$.

Now consider the case $\frac{b}{r_1}\frac{c}{r_1}(\frac{a}{s_1})^2\frac{d}{s_2}\frac{e}{s_3}=6^4t^11^1, t\in\{2,3\}$, then $M=r_1^2M_1=36r_1^2$.
Since $\gcd\{\frac{b}{r_1},\frac{c}{r_1}\}=1$, we may assume, without loss of generality, that either $\frac{b}{r_1}=1,\frac{c}{r_1}=t$ or
$\frac{b}{r_1}=1, \frac{c}{r_1}=6$. Given that $r_1>1$, we have $6r_1\not\in T$.
Applying Lemmas \ref{lem3.2.1} and \ref{jj}, we obtain either  MOA$((6r_1)^2;(6r_1)^1r_1^1(6s_1)^2(6s_2)^1(ts_3)^1,2)$ or  MOA$((6r_1)^2;(tr_1)^1r_1^1(6s_1)^2(6s_2)^1(6s_3)^1,2)$. Each of these is precisely the desired MOA$(M;a^2bcde,2)$.

For the case $\frac{b}{r_1}\frac{c}{r_1}(\frac{a}{s_1})^2\frac{d}{s_2}\frac{e}{s_3}=(2t)^4t^12^1$, $t\in\{3,5\}$, then $M=(2tr_1)^2$. Since $\gcd\{\frac{b}{r_1},\frac{c}{r_1}\}=1$, we may assume, without loss of generality, that either
$\frac{b}{r_1}=2, \frac{c}{r_1}=t$, or $\frac{b}{r_1}=t, \frac{c}{r_1}=2$. Given that $r_1>1$, we have $2t r_1\not\in T$. Applying Lemmas \ref{lem3.2.1} and \ref{jj},there exists an MOA$((2tr_1)^2;(2r_1)^1(tr_1)^1s_1^2s_2^1s_3^1,2)$, which is precisely the desired MOA$(M;a^2bcde,2)$.

For the case $\frac{b}{r_1}\frac{c}{r_1}(\frac{a}{s_1})^2\frac{d}{s_2}\frac{e}{s_3}=2^53^1$. Then $M=12r_1^2$. Since $\gcd\{\frac{b}{r_1},\frac{c}{r_1}\}=1$, we may assume that either $\frac{b}{r_1}=3, \frac{c}{r_1}=2, \frac{a}{s_1}=2,\frac{d}{s_2}=2, \frac{e}{s_3}=2$, or $\frac{b}{r_1}=2, \frac{c}{r_1}=3, \frac{a}{s_1}=2,\frac{d}{s_2}=2, \frac{e}{s_3}=2$. For the first case, we have $b=3r_1$, $c=2r_1$, $a=2s_1$, $d=2s_2$, $e=2s_3$, with $\gcd\{r_1,a\}=\gcd\{r_1,2s_1\}=s_1$,
$\gcd\{r_1,d\}=\gcd\{r_1,2s_2\}=s_2$, $\gcd\{r_1,e\}=\gcd\{r_1,2s_3\}=s_3$. It follows that $\gcd\{b,a\}=\gcd\{3r_1,2s_1\}=s_1>1$,
$\gcd\{b,d\}=\gcd\{3r_1,2s_2\}=s_2>1$, $\gcd\{b,e\}=\gcd\{3r_1,2s_3\}=s_3>1$ and
$s_i\mid r_1$ for $i=1,2,3$. Hence $2r_1\not\in T$. By Lemmas \ref{lem3.2.1} and  \ref{jj}, there exists an
MOA$((2r_1)^2;(2r_1)^1r_1^1(2s_1)^2(2s_2)^1(2s_3)^1,2)$. Applying Lemma \ref{jb} with $n=3$ yields an MOA$(12r_1^2;(2r_1)^1(3r_1)^1(2s_1)^2(2s_2)^1(2s_3)^1,2))$, which is precisely MOA$(M;a^2bcde,2)$. The second case follows by an identical argument.

For the case $\frac{b}{r_1}\frac{c}{r_1}(\frac{a}{s_1})^2\frac{d}{s_2}\frac{e}{s_3}=2^4(4k+2)^11^1,\ k>0$, then $M=2(4k+2)r_1^2$. Since $\gcd\{\frac{b}{r_1},\frac{c}{r_1}\}=1$, we may assume, without loss of generality, that either $\frac{b}{r_1}=1,\frac{c}{r_1}=2$ or
$\frac{b}{r_1}=1, \frac{c}{r_1}=4k+2$. In either case, we have $r_1=b$. Because $\gcd\{b,x\}>1$ for each $x\in\{a,c,d,e\}$, and no three of $b,c,d,e$ are all equal, at least two of $s_1,s_2,s_3$ must be distinct. Consequently, $r_1$ has at least two distinct factors greater than $1$, which implies that $2r_1\not\in T$.
By Lemmas \ref{lem3.2.1} and \ref{jj}, there exists an MOA$((2r_1)^2;(2r_1)^1r_1^1(2s_1)^2(2s_2)^1(2s_3)^1,2)$. Applying Lemma \ref{jb} then yields an MOA$(2(4k+2)r_1^2;(2r_1)^1r_1^1(2s_1)^2(2s_2)^1((4k+2)s_3)^1,2)$,
or an MOA$(2(4k+2)r_1^2;((4k+2)r_1)^1r_1^1(2s_1)^2(2s_2)^1(2s_3)^1,2)$. Each of these is precisely the desired MOA$(M;a^2bcde,2)$.
\end{proof}

\begin{lem}\label{fff1}
Let $a,b,c,d,e$ be positive integers greater than $1$, $a\in T=\{2,3,4,6,10,22\}$, $ a\neq b,\ a\neq c,\ a\neq d,\ a\neq e$ and no three of $b,c,d,e$ are all equal.
If $\gcd\{x,y\}\in T$ for any $\{x,y\}\subseteq\{a,b,c,d,e\}$, then there exists an MOA$(M;a^2bcde,2)$.
\end{lem}
\begin{proof}
Suppose that $\max\{\gcd\{x,y\}:\{x,y\}\subseteq\{a,b,c,d,e\}\}=r_1$, with $r_1\in T$. Without loss of generality, let
$\gcd\{b,c\}=r_1$, the other cases follow by an analogous argument. Define $\gcd\{r_1,a\}=r_2$, $\gcd\{r_1,d\}=r_3$, $\gcd\{r_1,e\}=r_4$.
The proof is then completed by considering five cases according to the possible values of $r_1$.

(1) For $r_1\in\{2,3\}$. In this case, we have $\gcd\{x,y\}=r_1$ for every pair $\{x,y\}\subseteq\{a,b,c,d,e\}$. By Lemma \ref{aaa}, it follows that $M=\frac{a^2bcde}{r_1^4}$.
Lemma \ref{3.2.5} then guarantees the existence of an MOA$(\frac{cde}{r_1};r_1^3cde,2)$. Applying Lemma \ref{jb} successively with
$n=\frac{a}{r_1}$, $n=\frac{a}{r_1}$ and $n=\frac{b}{r_1}$, yields the desired MOA$(M;a^2bcde,2)$. The same argument applies when $\gcd\{x,y\}=r_1\in\{4,6,10,22\}$ for all pairs $\{x,y\}\subseteq\{a,b,c,d,e\}$, and in each of these cases it also yields an MOA$(M;a^2bcde,2)$.

(2) For $r_1=4$. First, note that $a\neq3$, otherwise, we would have $3\mid b$ and $3\mid c$, contradicting $\gcd\{b,c\}=4$. If $a\in\{2,10,22\}$, then clearly $r_2=2$. By Lemmas \ref{k=23}-\ref{k=5}, there exists an  MOA$(M';(\frac{bc}{r_1^2})^1(\frac{a}{2})^2\frac{d}{r_3}\frac{e}{r_4},2)$. Applying Lemma \ref{split} yields an MOA$(M';\frac{b}{r_1}\frac{c}{r_1}(\frac{a}{2})^2\frac{d}{r_3}\frac{e}{r_4},2)$. Also, by Lemmas \ref{ua=2346} and \ref{jj}, there exists an MOA$(r_1^2;r_1^2r_2^2r_3r_4,2)$. Then Lemma \ref{bbb} gives the desired MOA$(M;a^2bcde,2)$.
If $a=4$, then clearly $r_2=4$. When $r_3=r_4=4$, this case has already been handled, since $\gcd\{x,y\}=4$ for any $\{x,y\}\subseteq\{a,b,c,d,e\}$. When $r_2=r_3=4,r_4=2$, there exists an MOA$(M;a^2bcd(\frac{4e}{r_4}),2)$ because $\gcd\{a,\frac{4e}{r_4}\}=4$. Applying Lemma \ref{jj} then yields an MOA$(M;a^2bcde,2)$ since $e \mid  \frac{4e}{r_4}$. When $\{r_2,r_3,r_4\}\in\{\{4,2,2\},\{2,2,2\}\}$, Lemmas \ref{lem3.2.4} and \ref{jj} gives an MOA$(a^2;a^2r_1^2r_3r_4,2)$. Moreover, since at most two of $b,c,d,e$ are equal, we have
$\frac{b}{r_1}\frac{c}{r_1}\frac{d}{r_3}\frac{e}{r_4}\not\in\mathcal{G}=\{2^4,6^4\}\cup \{2^3d^1:d\equiv 2\pmod 4\}$. Thus, by Lemma \ref{k=4}, there exists an MOA$(M';1^11^1\frac{b}{r_1}\frac{c}{r_1}\frac{d}{r_3}\frac{e}{r_4},2)$. Consequently, Lemma~\ref{jb1} yields an MOA$(a^2M';a^2bcde,2)$. If $a=6$, then $r_2=2$. By Lemma \ref{k=5}, there exists an  MOA$(M';(\frac{bc}{r_1^2})^1(\frac{a}{2})^2\frac{d}{r_3}\frac{e}{r_4},2)$
since $(\frac{bc}{r_1^2})^1(\frac{a}{2})^2\frac{d}{r_3}\frac{e}{r_4}\not\in \mathcal{L}_5$ for otherwise we would obtain a contradiction. Furthermore, by Lemmas \ref{MOLS} and\ref{jj}, there exists an MOA$(r_1^2;r_1^2r_2^2r_3r_4,2)$. Applying Lemma \ref{bbb} then gives the desired MOA$(M;a^2bcde,2)$.

(3) For $r_1=6$, with $a\in\{2,3\}$. In this case, we have $r_2=a$. When $r_3=r_4=6$, write $b=6p_1, c=6p_2,d=6p_3,e=6p_4>6$, with $\gcd\{p_i,p_j\}=1$ for $1\leq i\neq j \leq 4$. Lemma \ref{3.2.5} gives an MOA$(6e;a^26^3e,2)$. Applying Lemma \ref{jb} successively with $n=p_1,p_2,p_3$, respectively, then yields the desired MOA$(M;a^2bcde,2)$. When $r_3,r_4$ are not both equal to $6$, an MOA$(M';\frac{b}{r_1}\frac{c}{r_1}(\frac{a}{r_2})^2\frac{d}{r_3}\frac{e}{r_4},2)$ exists by Lemmas \ref{k=23} and \ref{k=4},
and an  MOA$(r_1^2;r_1^2r_2^2r_3r_4,2)$ exists by Lemmas \ref{36100} and \ref{jj}. Applying Lemma \ref{bbb} then gives the desired MOA$(M;a^2bcde,2)$.

For $r_1=6$, with $a=4$. Then $r_2=2$. When $r_3=r_4=6$, write $b=6p_1, c=6p_2,d=6p_3,e=6p_4$, with $\gcd\{p_i,p_j\}=1$ for $1\leq i\neq j \leq 4$.
Moreover, at least one $p_i$  is greater than 1. Without loss of generality, assume $p_1>1$. By Lemma \ref{3.2.5}, there exists an MOA$(M';a^26^3b,2)$. Applying Lemma \ref{jb} successively with $n=p_2,p_3,p_4$, respectively, then yields the desired MOA$(M;a^2bcde,2)$. When $r_3,r_4$ are not both equal to $6$, then for any $\{x,y\}\subseteq\{b,c,\frac{6d}{r_3},\frac{6e}{r_4}\}$, we have $\gcd\{x,y\}=6$. Hence, by the previously established result for the case where all pairwise $\gcd$s are equal to $6$, there exists an
MOA$(M;a^2bc(\frac{6d}{r_3})(\frac{6e}{r_4}),2)$. Applying Lemma \ref{jj} then yields
an MOA$(M;abcdef,2)$ exists since $d\mid \frac{6d}{r_3}$ and $e\mid \frac{6e}{r_4}$.

For $r_1=6$, with $a=6$. Then $r_2=6$. When $r_3=r_4=6$, then this case has already been handled, since $\gcd\{x,y\}=6$ for every pair $\{x,y\}\subseteq\{a,b,c,d,e\}$. When $r_3,r_4$ are not both equal to $6$, then for any pair $\{x,y\}\subseteq\{a,b,c,\frac{6d}{r_3},\frac{6e}{r_4}\}$, we have $\gcd\{x,y\}=6$.
Thus, by the previously established result for the case where all pairwise gcds are equal to $6$, there exists an MOA$(M;a^2bc(\frac{6d}{r_3})(\frac{6e}{r_4}),2)$.
Applying Lemma \ref{jj} then yields the desired MOA$(M;abcdef,2)$ since $d\mid \frac{6d}{r_3}$ and $e\mid \frac{6e}{r_4}$.

For $r_1=6$, with $a\in\{10,22\}$, we have $r_2=2$. When $r_3=r_4=6$, write $b=6p_1, c=6p_2,d=6p_3,e=6p_4>6$, with $\gcd\{p_i,p_j\}=1$ for $1\leq i\neq j \leq 4$, and $\gcd\{p_\alpha,\frac{a}{2}\}=1$, $\alpha=1,2,3$. By Lemma \ref{3.2.5}, there exists an MOA$(M_0;a^26^3e,2)$, where $M_0=\operatorname{lcm}\{a^2,6a,6e,36,ae\}=\frac{9a^2p_4}{\gcd\{\frac{a}{2},p_4\}}$. Applying Lemma \ref{jb} successively with $n=p_1,p_2,p_3$, respectively, then yields the desired MOA$(M;a^2bcde,2)$. When $r_3,r_4$ are not both equal to $6$, then by Lemma \ref{k=5}, there exists an  MOA$(M';(\frac{b}{r_1}\frac{c}{r_1})(\frac{a}{r_2})^2\frac{d}{r_3}\frac{e}{r_4},2)$, and by Lemmas \ref{36100} and \ref{jj}, there exists an MOA$(r_1^2;r_1^2r_2^2r_3r_4,2)$. Therefore, Lemma \ref{bbb} gives the desired MOA$(M;a^2bcde,2)$.

(5) For $r_1\in\{10,22\}$. In this case, we must have $a\neq3$, otherwise, $3\mid b$ and $3\mid c$, contradicting $\gcd\{b,c\}=r_1$. If $a=2$, then clearly $r_2=2$. By Lemmas \ref{k=23}-\ref{k=5}, there exists an  MOA$(M';(\frac{b}{r_1}\frac{c}{r_1})(\frac{a}{2})^2\frac{d}{r_3}\frac{e}{r_4},2)$, and by Lemmas \ref{MOLS} and \ref{jj}, there exists an MOA$(r_1^2;r_1^2r_2^2r_3r_4,2)$. Applying Lemma \ref{bbb} then yields the desired
MOA$(M;a^2bcde,2)$. If $a\in\{4,6,10,22\}$, the proof is analogous to that given for the case $r_1=6$.
\end{proof}

By Lemmas \ref{ccc}-\ref{fff1}, we obtain the following theorem.

\begin{thm}\label{6cc}
Let $a,b,c,d,e$ be positive integers greater than $1$ with $a\neq b$, $a\neq c$, $a\neq d$ and $a\neq e$, and no three of $b,c,d,e$ are all equal. If $\gcd\{x,y\}>1$ for any $\{x,y\}\subseteq\{a,b,c,d,e\}$, then there exists an MOA$(M;a^2bcde,2)$.
\end{thm}

\section*{Case 6:  $a,b,c,d,e,f$ are pairwise distinct}

In this case, we consider the existence of MOA$(M;abcdef,2)$, where
$a,b,c,d,e$ and $f$ are pairwise distinct and $M=\operatorname{lcm}\{xy: \{x,y\}\subset \{a,b,c,d,e,f\}\}$, which runs through this section.

\begin{lem}\label{ddd}
Let $a,b,c,d,e$ and $f$ be pairwise distinct integers greater than $1$. Suppose that $\gcd\{x,y\}>1$ for any $\{x,y\}\subseteq \{a,b,c,d,e,f\}$ and there is a subset $\{x,y\}\subseteq \{a,b,c,d,e,f\}$ such that $\gcd\{x,y\}=r_1\not\in T=\{2,3,4,6,10,22\}$, then there exists an MOA$(M;abcdef,2)$.
\end{lem}
\begin{proof}
We assume, without loss of generality, that $\gcd\{a,f\}=r_1>1$, $r_1\not\in T$, and define $\gcd\{r_1,b\}=r_2$, $\gcd\{r_1,c\}=r_3$, $\gcd\{r_1,d\}=r_4$,
$\gcd\{r_1,e\}=r_5$. Let
\begin{center}
$\mathcal{T}=\{s^51^1:s\in\{2,3,6,10\}\}\cup\{2^41^2,6^41^2,6^32^23^1,6^32^13^2,6^43^11^1,6^42^11^1,6^43^12^1,2^53^1,10^45^12^1\}$
$\cup\{2^5(2k+1)^1,2^4(4k+2)^11^1,2^3(4k+2)^11^2:k\geq1\}$.
\end{center}
By Lemma \ref{lem3.2.1}, there exists an MOA$(r_1^2;r_1^6,2)$. We obtain an MOA$(r_1^2;r_1^2r_2r_3r_4r_5,2)$ by Lemma \ref{jj}. Since $\gcd\{\frac{a}{r_1},\frac{f}{r_1}\}=1$, by Lemmas \ref{k=5} and \ref{lem3.1}, there exists an MOA$(M';\frac{a}{r_1}\frac{f}{r_1}\frac{b}{r_2}\frac{c}{r_3}
\frac{d}{r_4}\frac{e}{r_5},2)$ for $\frac{a}{r_1}\frac{f}{r_1}\frac{b}{r_2}\frac{c}{r_3}\frac{d}{r_4}\frac{e}{r_5}
\not\in\mathcal{T}$, where $M'$ is as defined in Lemma \ref{bbb}. Consequently, by Lemmas \ref{jb1} and \ref{bbb}, we obtain the desired MOA$(M;abcdef,2)$. We next turn to the case where $\frac{a}{r_1}\frac{f}{r_1}\frac{b}{r_2}\frac{c}{r_3}\frac{d}{r_4}\frac{e}{r_5}\in\mathcal{T}$.

1. $\frac{a}{r_1}\frac{f}{r_1}\frac{b}{r_2}\frac{c}{r_3}\frac{d}{r_4}
\frac{e}{r_5}=s^51^1$ for $s\in\{2,3,6,10\}$, and $\frac{a}{r_1}\frac{f}{r_1}\frac{b}{r_2}\frac{c}{r_3}\frac{d}{r_4}
\frac{e}{r_5}=s^41^2$ for $s\in\{2,6\}$

By hypothesis, we have $\gcd\{\frac{a}{r_1},\frac{f}{r_1}\}=1$. Without loss of generality, assume $\frac{a}{r_1}=1$, $\frac{f}{r_1}=\frac{b}{r_2}=\frac{c}{r_3}=\frac{d}{r_4}=\frac{e}{r_5}=s$.
It then follows immediately that $a=r_1$, $f=sr_1=sa$, $b=sr_2$, $c=sr_3$, $d=sr_4$, $e=sr_5$. Consequently, $\gcd\{b,f\}=\gcd\{sr_2,sr_1\}=sr_2=b$, so $b\mid f$. Also, $\gcd\{b,a\}=\gcd\{b,r_1\}=r_2=\gcd\{sr_2,r_1\}>1$,
which implies $\gcd\{s,\frac{r_1}{r_2}\}=1$. Now define
$\gcd\{b,c\}=s_1$, $\gcd\{b,d\}=s_2$, $\gcd\{b,e\}=s_3$. Then we have
$\frac{b}{sr_2}\frac{f}{sr_2}\frac{a}{r_2}\frac{c}{s_1}\frac{d}{s_2}\frac{e}{s_3}
=1^1(\frac{r_1}{r_2})^2\frac{c}{s_1}\frac{d}{s_2}\frac{e}{s_3}$. For brevity, denote the right-hand side by $(*)$.

(1)\ Case $s=2$. Since $\gcd\{2,\frac{r_1}{r_2}\}=1$, it follows that either $(*)=1^1(\frac{r_1}{r_2})^2\frac{c}{s_1}\frac{d}{s_2}\frac{e}{s_3}= 1^1 3^5$ or $1^1(\frac{r_1}{r_2})^2\frac{c}{s_1}\frac{d}{s_2}\frac{e}{s_3}\not\in\mathcal{T}$. When $(*)=1^1(\frac{r_1}{r_2})^2\frac{c}{s_1}\frac{d}{s_2}\frac{e}{s_3}=1^13^5$, then $r_1=3r_2$. Note that $\gcd\{\frac{r_1}{r_3},\frac{c}{r_3}\}=\gcd\{\frac{r_1}{r_3},s\}=1$ since $\gcd\{r_1,c\}=r_3$. A direct computation gives
\begin{center}
$\gcd\{c,f\}=\gcd\{2r_3,2r_1\}=2r_3$,\ \ $\gcd\{c,a\}=\gcd\{2r_3,r_1\}=r_3$.
\end{center}
Now set $\gcd\{c,d\}=t_2$, $\gcd\{c,e\}=t_3$ and recall $\gcd\{c,b\}=\gcd\{b,c\}=s_1$. Then
\begin{center}
$c=3s_1=3\gcd\{b,c\}=3\gcd\{2r_2,2r_3\}=6 \gcd\{r_2,r_3\}=6 \gcd\{\frac{r_1}{3},r_3\}=2\gcd\{r_1,3r_3\}=2r_3$,
\end{center}
which implies $\gcd\{\frac{r_1}{r_3},3\}=1$. Consequently, $\frac{c}{2r_3}\frac{f}{2r_3}\frac{a}{r_3}\frac{b}{s_1}\frac{d}{t_2}\frac{e}{t_3}=1^1(\frac{r_1}{r_3})^2\frac{b}{s_1}\frac{d}{t_2}\frac{e}{t_3}$ $\not\in$ $\mathcal{T}$. Therefore, by Lemmas \ref{k=23}-\ref{k=5} and \ref{lem3.1}, there exists an MOA$(M_{2};\frac{c}{2r_3}\frac{f}{2r_3}\frac{a}{r_3}\frac{b}{s_1}
\frac{d}{t_2}\frac{e}{t_3},2)$, where
$M_{2}=\operatorname{lcm}\{xy:\{x,y\}\subseteq\{\frac{c}{2r_3},\frac{f}{2r_3},\frac{a}{r_3},\frac{b}{s_1},\frac{d}{t_2},\frac{e}{t_3}\}\}$.
If $2r_3\not\in T$, then by Lemmas \ref{lem3.2.1} and \ref{jj}, there exists an MOA$((2r_3)^2;(2r_3)^2r_3^1s_1^1t_2^1t_3^1,2)$. Consequently, applying Lemmas~\ref{jb1} and \ref{bbb} yields the desired MOA$(M;abcdef,2)$, where $M=(2r_3)^2M_{2}$. If $2r_3\in T$, then $r_3\in\{2,3,5,11\}$. We first consider the case $r_3=2$. Then $c=2r_3=4$ and $r_3=\gcd\{r_1,c\}=\gcd\{r_1,4\}=2$. Thus $r_3\mid  r_1$ but $2r_3\nmid r_1$. It is then easy to see that $r_1$ is even and $r_1\not\in T$, which implies $2r_1\not\in T$. By Lemma \ref{lem3.2.1}, there exists an MOA$((2r_1)^{2};(2r_1)^6,2)$. Applying Lemma \ref{jj} then yields an MOA$((2r_1)^{2};r_1^1(2r_1)^1(2r_2)^1(2r_3)^1(2r_4)^1(2r_5)^1,2)$, which is precisely the required MOA$(M;abcdef,2)$. The same argument applies, mutatis mutandis, to the cases $r_3=3,5,11$, and in each case we similarly obtain an MOA$(M;abcdef,2)$.

When $(\ast)\not\in\mathcal{T}$. By Lemmas \ref{k=5} and \ref{lem3.1}, an MOA$(M_{1};\frac{b}{2r_2}\frac{f}{2r_2}\frac{a}{r_2}\frac{c}{s_1}
\frac{d}{s_2}\frac{e}{s_3},2)$ exists, where
$M_{1}=\operatorname{lcm}\{xy:\{x,y\}\subseteq\{\frac{b}{2r_2},\frac{f}{2r_2},\frac{a}{r_2},\frac{c}{s_1},\frac{d}{s_2},\frac{e}{s_3}\}\}$.
If $2r_2\not\in T$, by Lemma \ref{lem3.2.1}, an MOA$((2r_2)^2;(2r_2)^6,2)$ exists. Applying Lemma \ref{jj} yields an MOA$((2r_2)^2;(2r_2)^2r_2^1s_1^1s_2^1s_3^1,2)$.
Then, by Lemma \ref{jb1}, we obtain an MOA$((2r_2)^2M_{1};abcdef,2)$, which, by Lemma \ref{bbb}, is precisely an MOA$(M;abcdef,2)$. If $2r_2\in T=\{2,3,4,6,10,22\}$, then since $r_2=\gcd\{a,b\}>1$, we have $r_2\in\{2,3,5,11\}$. We first consider the case $r_2=2$. Then $b=2r_2=4$, and $r_2=\gcd\{r_1,b\}=\gcd\{r_1,4\}=2$. Thus, $r_2\mid  r_1$ but $2r_2\nmid r_1$. It is then easy to see that $r_1$ is even and $r_1\not\in T$, which implies $2r_1\not\in T$. By Lemma \ref{lem3.2.1}, there exists an MOA$((2r_1)^{2};(2r_1)^6,2)$. Applying Lemma \ref{jj} then yields an MOA$((2r_1)^{2};r_1^1(2r_1)^1(2r_2)^1(2r_3)^1(2r_4)^1(2r_5)^1,2)$, which is precisely
MOA$(M;abcdef,2)$. The same reasoning applies analogously to the cases
$r_2=3,5, 11$, and in each case we likewise obtain the desired MOA$(M;abcdef,2)$.

(2) Case $s=3$. Since $\gcd\{3,\frac{r_1}{r_2}\}=1$, it follows that  either $(\ast)=1^1(\frac{r_1}{r_2})^2\frac{c}{s_1}\frac{d}{s_2}\frac{e}{s_3}\in \mathcal{U}$, where $\mathcal{U}=\{2^51^1,10^51^1,2^41^2\}\cup\{2^4(4k+2)^11^1,2^3(4k+2)^11^2:  k\geq1\}$, or $(\ast)\not\in \mathcal{T}$.

When $(\ast)\not\in\mathcal{T}$. By Lemmas \ref{k=5} and \ref{lem3.1}, there exists an MOA$(M_{1};\frac{b}{3r_2}\frac{f}{3r_2}\frac{a}{r_2}\frac{c}{s_1}
\frac{d}{s_2}\frac{e}{s_3},2)$, where $M_{1}=\operatorname{lcm}\{xy:\{x,y\}\subseteq\{\frac{b}{3r_2},\frac{f}{3r_2},\frac{a}{r_2},\frac{c}{s_1},\frac{d}{s_2},\frac{e}{s_3}\}\}$.
If $3r_2\not\in T$, by Lemma \ref{lem3.2.1}, there exists an MOA$((3r_2)^2;(3r_2)^6,2)$. Applying Lemma \ref{jj} yields an MOA$((3r_2)^2;(3r_2)^2r_2^1s_1^1s_2^1s_3^1,2)$. Thus, by Lemmas \ref{jb1} and \ref{bbb}, there exists an MOA$(M;abcdef,2)$, where $M=(3r_2)^2M_{1}$. If $3r_2\in T$. Since $r_2>1$, we must have $b=3r_2=6$ and hence $r_2=2$.
It is clear that $a=r_1$, $f=3r_1$, $c=3r_3$, $d=3r_4$, $e=3r_5$ and $r_1=\gcd\{a,f\}>1$. We claim that $r_1>2$. Indeed, if $r_1=2$, then  $\gcd\{a,c\}=\gcd\{r_1,c\}=r_3>1$, $\gcd\{a,d\}=\gcd\{r_1,d\}=r_4>1$, which forces $r_3=r_4=2$, contradicting the fact that $c\neq d$. Hence $r_1>2$. It follows that $3r_1\not\in T$. By Lemma \ref{lem3.2.1}, there exists an MOA$((3r_1)^2;(3r_1)^6,2)$. Applying Lemma \ref{jj} yields an MOA$((3r_1)^2;(3r_1)^1r_1^1(3r_2)^1(3r_3)^1(3r_4)^1(3r_5)^1,2)$, that is, an MOA$(M;abcdef,2)$.

When $(\ast)\in\mathcal{U}$, we have $r_1=wr_2$, where $w\in\{2,10\}$. A straightforward calculation gives $\gcd\{c,f\}=\gcd\{3r_3,3r_1\}=3r_3$, $\gcd\{c,a\}=\gcd\{c,r_1\}=r_3$. Since $\gcd\{c,r_1\}=\gcd\{3r_3, r_1\}=r_3$, we also have $\gcd\{\frac{r_1}{r_3},3\}=1$. Moreover, at least one of \(\frac{c}{s_1},\frac{d}{s_2},\frac{e}{s_3}\) belongs to $\{2,10\}$.
Suppose, without loss of generality, that $\frac{c}{s_1}\in\{2,10\}$. Define $\gcd\{c,d\}=t_2$, $\gcd\{c,e\}=t_3$. Then
\begin{center}
$c=3r_3=ws_1=w \gcd\{b,c\}=w \gcd\{3r_2,3r_3\}=3\gcd\{r_1,wr_3\}$, $w\in\{2,10\}$,
\end{center}
It follows that $\gcd\{\frac{r_1}{r_3},w\}=1$. Consequently, $\frac{c}{3r_3}\frac{f}{3r_3}\frac{a}{r_3}\frac{b}{s_1}\frac{d}{t_2}\frac{e}{t_3}=1^1(\frac{r_1}{r_3})^2\frac{b}{s_1}\frac{d}{t_2}\frac{e}{t_3}\not\in \mathcal{T}$. Thus, by Lemmas \ref{k=23}-\ref{k=5} and \ref{lem3.1}, there exists an MOA$(M_{2};\frac{c}{3r_3}\frac{f}{3r_3}\frac{a}{r_3}\frac{b}{s_1}\frac{d}{t_2}\frac{e}{t_3},2)$, where $M_{2}=\operatorname{lcm}\{xy:\{x,y\}\subseteq\{\frac{c}{3r_3},\frac{f}{3r_3},\frac{a}{r_3},\frac{b}{s_1},\frac{d}{t_2},\frac{e}{t_3}\}\}$.
If $3r_3\not\in T$, then there exists an MOA$((3r_3)^2;(3r_3)^2r_3^1s_1^1t_2^1t_3^1,2)$.
By Lemmas \ref{jb1} and \ref{bbb}, an MOA$(M;abcdef,2)$ exists, where $M=(3r_3)^2M_{2}$. If $3r_3\in T$, then since $r_3>1$, we must have $c=3r_3=6$,
hence $r_3=2$. It is clear that $a=r_1$, $f=3r_1$, $b=3r_2$, $d=3r_4$, $e=3r_5$ and $r_1=\gcd\{a,f\}>1$. We claim that $r_1>2$. Indeed, if $r_1=2$, then  $\gcd\{a,b\}=\gcd\{r_1,b\}=r_2>1$, $\gcd\{a,d\}=\gcd\{r_1,d\}=r_4>1$,
which forces $r_2=r_4=2$, contradicting $b\neq d$.
Therefore, $r_1>2$, and consequently $3r_1\not\in T$. By Lemmas \ref{lem3.2.1} and \ref{jj}, there exists an MOA$((3r_1)^2;(3r_1)^1r_1^1(3r_2)^1(3r_3)^1(3r_4)^1(3r_5)^1,2)$,
which is precisely the required MOA$(M;abcdef,2)$.

(3) Case $s=6$. Clearly, we have $\gcd\{6,\frac{r_1}{r_2}\}=1$ and $(\ast)\not\in \mathcal{T}$. By Lemmas \ref{k=23}-\ref{k=5} and \ref{lem3.1},
an MOA$(M_{1};\frac{b}{6r_2}\frac{f}{6r_2}\frac{a}{r_2}\frac{c}{s_1}
\frac{d}{s_2}\frac{e}{s_3},2)$ exists, where $M_{1}=\operatorname{lcm}\{xy:\{x,y\}\subseteq\{\frac{b}{6r_2},
\frac{f}{6r_2},\frac{a}{r_2},\frac{c}{s_1},\frac{d}{s_2},\frac{e}{s_3}\}\}$.
Since $r_2=\gcd\{a,b\}>1$ and $6r_2\not\in T$, Lemma \ref{lem3.2.1} gives an MOA$((6r_2)^2;(6r_2)^6,2)$. Applying Lemma \ref{jj} yields an MOA$((6r_2)^2;(6r_2)^2r_2^1s_1^1s_2^1s_3^1,2)$. Therefore, by Lemmas \ref{jb1} and \ref{bbb}, there exists an MOA$(M;abcdef,2)$, where $M=(6r_2)^2M_{1}$.

(4) Case $s=10$. Since $\gcd\{10,\frac{r_1}{r_2}\}=1$, it follows that either $(\ast)=1^1(\frac{r_1}{r_2})^2\frac{c}{s_1}\frac{d}{s_2}\frac{e}{s_3}=1^1 3^5$ or $(\ast)\not\in\mathcal{T}$. When $(\ast)\not\in\mathcal{T}$. By Lemmas \ref{k=23}-\ref{k=5} and \ref{lem3.1}, an MOA$(M_{1};\frac{b}{10r_2}\frac{f}{10r_2}\frac{a}{r_2}\frac{c}{s_1}
\frac{d}{s_2}\frac{e}{s_3},2)$ exists, where $M_{1}=\operatorname{lcm}\{xy:\{x,y\}\subseteq\{\frac{b}{10r_2},\frac{f}{10r_2},\frac{a}{r_2},\frac{c}{s_1},\frac{d}{s_2},\frac{e}{s_3}\}\}$.
Since $r_2=\gcd\{a,b\}>1$, we have $10r_2\not\in T$. By Lemma \ref{lem3.2.1}, there exists an MOA$((10r_2)^2;(10r_2)^6,2)$, and by Lemma \ref{jj}, there exists an MOA$((10r_2)^2;(10r_2)^2r_2^1s_1^1s_2^1s_3^1,2)$. Applying Lemmas \ref{jb1} and \ref{bbb} yields an MOA$(M;abcdef,2)$, where $M=(10r_2)^2M_{1}$.

When $(\ast)=1^1(\frac{r_1}{r_2})^2\frac{c}{s_1}\frac{d}{s_2}\frac{e}{s_3}=1^1 3^5$. It is easy to see that $r_1=3r_2$. Moreover, $\gcd\{c,f\}=\gcd\{10r_3,10r_1\}=10r_3$, $\gcd\{10r_3,a\}=\gcd\{c,a\}=r_3>1$. Let $\gcd\{c,d\}=t_2$, $\gcd\{c,e\}=t_3$.
Then $c=10r_3=3s_1=3 \gcd\{b,c\}=3 \gcd\{10r_2,10r_3\}=10\gcd\{r_1,3r_3\}$.
It follows that $\gcd\{\frac{r_1}{r_3},3\}=1$. Furthermore, from $\gcd\{r_1,c\}=\gcd\{r_1,10r_3\}=r_3$, we also obtain $\gcd\{\frac{r_1}{r_3},10\}=1$.
Consequently, $\frac{c}{10r_3}\frac{f}{10r_3}\frac{a}{r_3}\frac{b}{t_1}\frac{d}{t_2}\frac{e}{t_3}=1^1(\frac{r_1}{r_3})^2\frac{b}{t_1}\frac{d}{t_2}\frac{e}{t_3}\not\in \mathcal{T}$. By Lemmas \ref{k=23}-\ref{k=5} and \ref{lem3.1}, there exists an MOA$(M_{2};\frac{c}{10r_3}\frac{f}{10r_3}\frac{a}{r_3}\frac{b}{s_1}\frac{d}{t_2}\frac{e}{t_3},2)$,
where $M_{2}=\operatorname{lcm}\{xy:\{x,y\}\subseteq\{\frac{c}{10r_3},\frac{f}{10r_3},\frac{a}{r_3},\frac{b}{t_1},\frac{d}{t_2},\frac{e}{t_3}\}\}$.
Since $\gcd\{c,a\}=r_3>1$, we have $10r_3\not\in T$, by Lemmas \ref{lem3.2.1} and \ref{jj}, there exists an MOA$((10r_3)^2;(10r_3)^2r_3^1s_1^1t_2^1t_3^1,2)$. Applying Lemmas \ref{jb1} and \ref{bbb}, an MOA$(M;abcdef,2)$ exists,
where $M=(10r_3)^2M_{2}$.

The proof for the case $\frac{a}{r_1}\frac{f}{r_1}\frac{b}{r_2}\frac{c}{r_3}\frac{d}{r_4}
\frac{e}{r_5}=s^41^2$, $s\in\{2,6\}$, is analogous to the above argument.

2. $\frac{a}{r_1}\frac{f}{r_1}\frac{b}{r_2}\frac{c}{r_3}\frac{d}{r_4}\frac{e}{r_5}
\in\{2^53^1\}\cup\{2^5(2k+1)^1:k>0\}$.

We consider only the case
$\frac{a}{r_1}\frac{f}{r_1}\frac{b}{r_2}\frac{c}{r_3}\frac{d}{r_4}\frac{e}{r_5}
=2^53^1$, the other case is handled analogously.

We assume, without loss of generality, that $\frac{a}{r_1}=3, \frac{f}{r_1}=\frac{b}{r_2}=\frac{c}{r_3}=\frac{d}{r_4}=\frac{e}{r_5}=2$.
Consequently, $a=3r_1, f=2r_1, b=2r_2, c=2r_3, d=2r_4, e=2r_5$. Moreover, $r_1,r_2,r_3,r_4$ and $r_5$ are pairwise distinct, since $b,c,d,e,f$ are pairwise distinct.
Hence, there exists at least one element $u\in\{r_2,r_3,r_4,r_5,2r_2,
2r_3,2r_4,2r_5\}$ such that $u\not\in T$. Let $u=2r_2\not\in T$, and set $\gcd\{2r_2,c\}=s_1$, $\gcd\{2r_2,d\}=s_2$, $\gcd\{2r_2,e\}=s_3$. Then, by Lemmas \ref{lem3.2.1} and \ref{jj}, there exists an MOA$((2r_2)^2;(2r_2)^2r_2^1s_1^1s_2^1s_3^1,2)$. Now observe that $\gcd\{b,f\}=\gcd\{2r_2,2r_1\}=2r_2$ and $\gcd\{2r_2,a\}=\gcd\{2r_2,3r_1\}=r_2$. It follows that $\gcd\{\frac{r_1}{r_2},2\}=1$. Thus, $\frac{r_1}{r_2}> 1$ and is odd. Therefore, $\frac{b}{2r_2}\frac{f}{2r_2}\frac{a}{r_2}\frac{c}{s_1}
\frac{d}{s_2}\frac{e}{s_3}=1\frac{r_1}{r_2}\frac{3r_1}{r_2}\frac{c}{s_1}\frac{d}{s_2}\frac{e}{s_3}
\not\in \mathcal{T}$. By Lemmas \ref{k=23}-\ref{k=5} and \ref{lem3.1}, there exists an MOA$(M_{1};\frac{b}{2r_2}\frac{f}{2r_2}\frac{a}{r_2}
\frac{c}{s_1}\frac{d}{s_2}\frac{e}{s_3},2)$,
where $M_1=\operatorname{lcm}\{xy:\{x,y\}\subseteq\{\frac{b}{2r_2},
\frac{f}{2r_2},\frac{a}{r_2},\frac{c}{s_1},\frac{d}{s_2},\frac{e}{s_3}\}\}$. Finally, applying Lemmas \ref{jb1} and \ref{bbb}, we obtain an MOA$(M;abcdef,2)$, where $M=(2r_2)^2M_{1}$.

3. $\frac{a}{r_1}\frac{f}{r_1}\frac{b}{r_2}\frac{c}{r_3}\frac{d}{r_4}
\frac{e}{r_5}\in\{2^4(4k+2)^11^1, 2^3(4k+2)^11^2:k>0\}$.

We consider only the case $\frac{a}{r_1}\frac{f}{r_1}\frac{b}{r_2}\frac{c}{r_3}\frac{d}{r_4}
\frac{e}{r_5}\in\{2^4(4k+2)^11^1:k>0\}$. The proof for the case $\frac{a}{r_1}\frac{f}{r_1}\frac{b}{r_2}\frac{c}{r_3}\frac{d}{r_4}
\frac{e}{r_5}\in\{2^3(4k+2)^11^2:k>0\}$ is analogous and is therefore omitted.

Since $\gcd\{\frac{a}{r_1},\frac{f}{r_1}\}=1$, we may assume without loss of generality that either $\frac{a}{r_1}=1,\frac{f}{r_1}=4k+2$ or $\frac{a}{r_1}=1,\frac{f}{r_1}=2$. We first treat the former case, namely $\frac{a}{r_1}=1,\frac{f}{r_1}=4k+2$,
$\frac{b}{r_2}=\frac{c}{r_3}=\frac{d}{r_4}=\frac{e}{r_5}=2$. Thus, $a=r_1, f=(4k+2)r_1, b=2r_2,  c=2r_3,  d=2r_4,  e=2r_5$.
A direct computation gives $\gcd\{b,f\}=2r_2$, $\gcd\{2r_2,a\}=\gcd\{b,r_1\}=\gcd\{2r_2,r_1\}=r_2$.
Hence $\gcd\{\frac{r_1}{r_2},2\}=1$, so both $\frac{r_1}{r_2}$ and $\frac{(2k+1)r_1}{r_2}$ are odd.
Define $\gcd\{2r_2,c\}=s_1$, $\gcd\{2r_2,d\}=s_2$, $\gcd\{2r_2,e\}=s_3$.
It is clear that $\frac{b}{2r_2}\frac{f}{2r_2}\frac{a}{r_2}\frac{c}{s_1}\frac{d}{s_2}\frac{e}{s_3}=1\frac{r_1}{r_2}\frac{(2k+1)r_1}{r_2}\frac{c}{s_1}\frac{d}{s_2}\frac{e}{s_3}
\not\in \mathcal{T}$.
By Lemmas \ref{k=4}-\ref{k=5} and \ref{lem3.1}, there exists an
MOA$(M_{1};\frac{b}{2r_2}\frac{f}{2r_2}\frac{a}{r_2}\frac{c}{s_1}\frac{d}{s_2}\frac{e}{s_3},2)$,
where $M_1=\operatorname{lcm}\{xy:\{x,y\}\subseteq\{\frac{b}{2r_2},\frac{f}{2r_2},
\frac{a}{r_2},\frac{c}{s_1},\frac{d}{s_2},\frac{e}{s_3}\}\}$. We now distinguish two subcases according to whether $2r_2\in T$. When $2r_2\not\in T$,
by Lemmas \ref{lem3.2.1} and \ref{jj}, there exists an MOA$((2r_2)^2;(2r_2)^2r_2^1s_1^1s_2^1s_3^1,2)$. Then, applying Lemmas \ref{jb1} and \ref{bbb}, we obtain an MOA$(M;abcdef,2)$, where $M=(2r_2)^2M_{1}$.
When $2r_2\in T$, since $r_2=\gcd\{a,b\}>1$, we must have $r_2\in\{2,3,5,11\}$.
If $r_2=2$, then $b=2r_2=4$ and $\gcd\{r_1,b\}=\gcd\{r_1,4\}=2$, so $2r_2\nmid r_1$.
Moreover, because $b,c,d,e$ are pairwise distinct and $\frac{b}{r_2}=\frac{c}{r_3}=\frac{d}{r_4}=\frac{e}{r_5}=2$, it follows that $r_2,r_3,r_4$ and $r_5$ are pairwise distinct and $r_i\mid r_1$ for $i\in\{2,3,4,5\}$. Hence $2r_1\not\in T$. By Lemmas \ref{lem3.2.1} and \ref{jj}, there exists an MOA$((2r_1)^{2};(2r_2)^1(2r_3)^1(2r_4)^1(2r_5)^1r_1^1(2r_1)^1,2)$.
Applying Lemma \ref{jb} with $n=2k+1$ yields an
MOA$((2r_1)^{2}(2k+1);(2r_2)^1(2r_3)^1(2r_4)^1(2r_5)^1r_1^1((4k+2)r_1)^1,2)$, that is, an MOA$(M;abcdef,2)$. For the remaining cases $r_2=3,5,11$, the same argument applies mutatis mutandis, and we again obtain the desired MOA$(M;abcdef,2)$.

We now turn to the second case, where $\frac{a}{r_1}=1$, $\frac{f}{r_1}=2$, and assume $\frac{b}{r_2}=(4k+2)$, $\frac{c}{r_3}=\frac{d}{r_4}=\frac{e}{r_5}=2$.
Then clearly $a=r_1$, $f=2r_1$, $b=(4k+2)r_2$, $c=2r_3$, $d=2r_4$, $e=2r_5$.
Consequently, $\gcd\{b,f\}=\gcd\{(4k+2)r_2,2r_1\}=2r_2$ and $\gcd\{2r_2,a\}=\gcd\{2r_2,r_1\}=r_2>1$. It follows that $\gcd\{2k+1,\frac{r_1}{r_2}\}=1$ and $\gcd\{\frac{r_1}{r_2},2\}=1$.
Define $\gcd\{2r_2,c\}=s_1$, $\gcd\{2r_2,d\}=s_2$, $\gcd\{2r_2,e\}=s_3$.
Then $\frac{b}{2r_2}\frac{f}{2r_2}\frac{a}{r_2}\frac{c}{s_1}\frac{d}{s_2}\frac{e}{s_3}
=(2k+1)^1(\frac{r_1}{r_2})^2\frac{c}{s_1}\frac{d}{s_2}\frac{e}{s_3}
\not\in \mathcal{T}$.
By Lemmas \ref{k=4}-\ref{k=5} and \ref{lem3.1}, there exists an MOA$(M_{1};\frac{b}{2r_2}\frac{f}{2r_2}\frac{a}{r_2}\frac{c}{s_1}\frac{d}{s_2}\frac{e}{s_3},2)$,
where $M_1=\operatorname{lcm}\{xy:\{x,y\}\subseteq
\{\frac{b}{2r_2},\frac{f}{2r_2},\frac{a}{r_2},\frac{c}{s_1},\frac{d}{s_2},
\frac{e}{s_3}\}\}$. We again distinguish two subcases. When $2r_2\not\in T$, by Lemmas \ref{lem3.2.1} and \ref{jj}, there exists an MOA$((2r_2)^2;(2r_2)^2r_2^1s_1^1s_2^1s_3^1,2)$. Applying Lemmas \ref{jb1} and \ref{bbb}, we obtain an MOA$(M;abcdef,2)$, where $M=(2r_2)^2M_{1}$.
When $2r_2\in T$, since $r_2=\gcd\{a,b\}>1$, we must have $r_2\in\{2,3,5,11\}$.
If $r_2=2$, then $b=2r_2=4$, and $\gcd\{r_1,b\}=\gcd\{r_1,4\}=2$, so $2r_2\nmid r_1$.
Moreover, because $c,d,e$ are pairwise distinct and $\frac{c}{r_3}=\frac{d}{r_4}=\frac{e}{r_5}=2$, it follows that
$r_3,r_4,r_5$ are pairwise distinct and $r_i\mid r_1$ for $i\in\{3,4,5\}$.
Hence $2r_1\not\in T$. By Lemmas \ref{lem3.2.1} and \ref{jj}, there exists an MOA$((2r_1)^{2};(2r_2)^1(2r_3)^1(2r_4)^1(2r_5)^1r_1^1(2r_1)^1,2)$.
Applying Lemma \ref{jb} with $n=2k+1$ yields an
MOA$((2r_1)^{2}(2k+1);((4k+2)r_2)^1(2r_3)^1(2r_4)^1(2r_5)^1r_1^1(2r_1)^1,2)$, which is precisely an MOA$(M;abcdef,2)$. For the remaining cases $r_2=3,5,11$, the same argument applies mutatis mutandis, and we again obtain the desired MOA$(M;abcdef,2)$.

4. $\frac{a}{r_1}\frac{f}{r_1}\frac{b}{r_2}\frac{c}{r_3}\frac{d}{r_4}\frac{e}{r_5}
\in\{6^42^11^1,6^43^11^1,6^32^23^1,6^32^13^2\}$.

When $\frac{a}{r_1}\frac{f}{r_1}\frac{b}{r_2}\frac{c}{r_3}\frac{d}{r_4}\frac{e}{r_5}
=6^42^11^1$, the proof is analogous to that for the case $\frac{a}{r_1}\frac{f}{r_1}\frac{b}{r_2}\frac{c}{r_3}\frac{d}{r_4}\frac{e}{r_5}
=2^46^11^1$.

Now consider the case $\frac{a}{r_1}\frac{f}{r_1}\frac{b}{r_2}\frac{c}{r_3}\frac{d}{r_4}\frac{e}{r_5}
=6^43^11^1$. Since $\gcd\{\frac{a}{r_1},\frac{f}{r_1}\}=1$, we may assume without loss of generality that either $\frac{a}{r_1}=1, \frac{f}{r_1}=3$ or $\frac{a}{r_1}=1, \frac{f}{r_1}=6$.

(1) Consider the subcase $\frac{a}{r_1}=1, \frac{f}{r_1}=3$, with $\frac{b}{r_2}=\frac{c}{r_3}=\frac{d}{r_4}=\frac{e}{r_5}=6$. Thus,
$a=r_1, f=3r_1, b=6r_2,\ c=6r_3, d=6r_4, e=6r_5$. Since $b,c,d,e$ are pairwise distinct integers, it follows that $r_2,r_3,r_4$ and $r_5$ are also pairwise distinct. Therefore, there exists at least one element $u\in\{3r_2,3r_3,3r_4,3r_5\}$ such that $u\not\in T$. Without loss of generality, let $u=3r_2\not\in T$. Clearly, $\gcd\{b,f\}=\gcd\{6r_2,3r_1\}=3r_2$
and $\gcd\{3r_2,a\}=\gcd\{3r_2,r_1\}=r_2$. Moreover, since $\gcd\{a,b\}=\gcd\{r_1, 6r_2\}=r_2$, we have $\gcd\{\frac{r_1}{r_2},6\}=1$. Define
$\gcd\{3r_2,c\}=s_1$, $\gcd\{3r_2,d\}=s_2$, $\gcd\{3r_2,e\}=s_3$. Then $\frac{b}{3r_2}\frac{f}{3r_2}\frac{a}{r_2}\frac{c}{s_1}\frac{d}{s_2}\frac{e}{s_3}
=2^1(\frac{r_1}{r_2})^2\frac{c}{s_1}\frac{d}{s_2}\frac{e}{s_3}
\not\in \mathcal{T}$ or $2^1(\frac{r_1}{r_2})^2\frac{c}{s_1}\frac{d}{s_2}\frac{e}{s_3}
\in \{2^41^2,6^32^13^2\} \cup \{2^3(4k+2)^11^2:k\geq1\}$.

In the former case, namely $2^1(\frac{r_1}{r_2})^2\frac{c}{s_1}\frac{d}{s_2}\frac{e}{s_3}\not\in \mathcal{T}$. By Lemmas \ref{k=23}-\ref{k=5} and \ref{lem3.1}, there exists an MOA$(M_{1};\frac{b}{3r_2}\frac{f}{3r_2}\frac{a}{r_2}\frac{c}{s_1}\frac{d}{s_2}\frac{e}{s_3},2)$, where $M_{1}=\operatorname{lcm}\{xy:\{x,y\}\subseteq\{\frac{b}{3r_2},\frac{f}{3r_2},\frac{a}{r_2},\frac{c}{s_1},\frac{d}{s_2},\frac{e}{s_3}\}\}$. Since $3r_2\not\in T$, by Lemmas \ref{lem3.2.1} and \ref{jj}, there exists an MOA$((3r_2)^2;(3r_2)^2r_2^1s_1^1s_2^1s_3^1,2)$. Applying Lemmas \ref{jb1} and \ref{bbb}, there exists an MOA$(M;abcdef,2)$, where $M=(3r_2)^2M_{1}$.

Now consider the case where $2^1(\frac{r_1}{r_2})^2\frac{c}{s_1}\frac{d}{s_2}\frac{e}{s_3}\in\{2^41^2\} \cup \{2^3(4k+2)^11^2:k\geq1\}$. Then necessarily $r_1=r_2$. Moreover, we have $\gcd\{b,c\}=\gcd\{6r_1,6r_3\}=6r_3$, $\gcd\{c,a\}=\gcd\{6r_3,r_1\}=r_3$,
$\gcd\{c,f\}=3r_3$. It follows that $\gcd\{6,\frac{r_1}{r_3}\}=1$. Define $\gcd\{c,d\}=t_2$, $\gcd\{c,e\}=t_3$. Then $\frac{b}{6r_3}\frac{c}{6r_3}\frac{a}{r_3}\frac{f}{3r_3}\frac{d}{t_2}\frac{e}{t_3}=
1^1(\frac{r_1}{r_3})^3\frac{d}{t_2}\frac{e}{t_3}\not\in\mathcal{T}$.
By Lemmas \ref{k=23}-\ref{k=5} and \ref{lem3.1}, an
MOA$(M_2;\frac{c}{6r_3}\frac{b}{6r_3}\frac{a}{r_3}\frac{f}{3r_3}\frac{d}{t_2}\frac{e}{t_3},2)$ exists,
where $M_{2}=\operatorname{lcm}\{xy:\{x,y\}\subseteq\{\frac{c}{6r_3},
\frac{b}{6r_3},\frac{a}{r_3},\frac{f}{3r_3},\frac{d}{t_2},\frac{e}{t_3}\}\}$.
Since $r_3>1$, we have $6r_3\not\in T$. By Lemmas \ref{lem3.2.1} and \ref{jj}, there exists an MOA$((6r_3)^2;(6r_3)^2(3r_3)^1r_3^1t_2^1t_3^1,2)$. Applying Lemmas \ref{jb1} and \ref{bbb} yields an MOA$(M;abcdef,2)$, where $M=(6r_3)^2M_{2}$.

Consider the case $2^1(\frac{r_1}{r_2})^2\frac{c}{s_1}\frac{d}{s_2}\frac{e}{s_3}=2^13^26^3$. Then necessarily $r_1=3r_2$, $s_1=r_3$, $s_2=r_4$, $s_3=r_5$. Moreover, we have $\gcd\{b,f\}=\gcd\{2r_1,3r_1\}=r_1$, $\gcd\{r_1,a\}=r_1$,
$\gcd\{r_1,c\}=r_3$, $\gcd\{r_1,d\}=r_4$, $\gcd\{r_1,e\}=r_5$. Consequently, $\frac{b}{r_1}\frac{f}{r_1}\frac{a}{r_1}\frac{c}{r_3}\frac{d}{r_4}\frac{e}{r_5}=2^13^11^16^3$.
By Lemma \ref{36100}, there exists an MOA$(36;2^13^11^16^3,2)$, and by Lemma \ref{3.2.5}, there exists an MOA$(r_1^2;r_1^3r_3r_4r_5,2)$. Applying Lemmas \ref{jb1} and \ref{bbb} yields an MOA$(M;abcdef,2)$, where $M=36M_{2}$.

(2) Consider the subcase $\frac{a}{r_1}=1, \frac{f}{r_1}=6$. We assume
$\frac{e}{r_5}=3$, $\frac{b}{r_2}=\frac{c}{r_3}=\frac{d}{r_4}=6$.
Then $a=r_1$, $f=6r_1$, $b=6r_2$, $c=6r_3$, $d=6r_4$, $e=3r_5$. Since $r_2=\gcd\{r_1,b\}=\gcd\{r_1,6r_2\}=\gcd\{a,b\}>1$, we have
$\gcd\{\frac{r_1}{r_2},6\}=1$, and consequently $6r_2\not\in T$. Define
$\gcd\{6r_2,c\}=t_1$, $\gcd\{6r_2,d\}=t_2$, $\gcd\{6r_2,e\}=t_3$. Then
$\frac{f}{6r_2}\frac{b}{6r_2}\frac{a}{r_2}\frac{c}{t_1}\frac{d}{t_2}\frac{e}{t_3}
=1^1(\frac{r_1}{r_2})^2\frac{c}{t_1}\frac{d}{t_2}\frac{e}{t_3}
\not\in\mathcal{T}$.
By Lemmas \ref{k=23}-\ref{k=5} and \ref{lem3.1}, there exists an
MOA$(M_{1};\frac{f}{6r_2}\frac{b}{6r_2}\frac{a}{r_2}\frac{c}{t_1}\frac{d}{t_2}\frac{e}{t_3},2)$,
where $M_1=\operatorname{lcm}\{xy:\{x,y\}\subseteq\{\frac{f}{6r_2},\frac{b}{6r_2},\frac{a}{r_2},\frac{c}{t_1},\frac{d}{t_2},\frac{e}{t_3}\}\}$,
and by Lemmas \ref{lem3.2.1} and \ref{jj}, there exists an MOA$((6r_2)^2;(6r_2)^2r_2^1t_1^1t_2^1t_3^1,2)$. Applying Lemmas \ref{jb1} and  \ref{bbb}, there exists an MOA$(M;abcdef,2)$, where $M=(6r_2)^2M_{1}$.

When $\frac{a}{r_1}\frac{f}{r_1}\frac{b}{r_2}\frac{c}{r_3}\frac{d}{r_4}\frac{e}{r_5}
\in\{6^32^23^1,6^32^13^2\}$. the proofs are analogous to those of the preceding cases.

5. $\frac{a}{r_1}\frac{f}{r_1}\frac{b}{r_2}\frac{c}{r_3}\frac{d}{r_4}\frac{e}{r_5}
\in\{6^43^12^1,10^45^12^1\}$.

We consider only the case $\frac{a}{r_1}\frac{f}{r_1}\frac{b}{r_2}\frac{c}{r_3}\frac{d}{r_4}\frac{e}{r_5}
=2^13^16^4$. The proof for the other case $10^45^12^1$ is analogous.
Since $\gcd\{\frac{a}{r_1},\frac{f}{r_1}\}=1$, we may assume $\frac{a}{r_1}=2$, $\frac{f}{r_1}=3$. Then clearly $\frac{b}{r_2}=\frac{c}{r_3}=\frac{d}{r_4}=\frac{e}{r_5}=6$, so that
$a=2r_1, f=3r_1, b=6r_2, c=6r_3, d=6r_4, e=6r_5$. Since $b,c,d,e$ are pairwise distinct
and $\frac{b}{r_2}\frac{c}{r_3}\frac{d}{r_4}\frac{e}{r_5}=6^4$, it follows that
$r_2,r_3,r_4,r_5$ are pairwise distinct. Hence there must exist distinct indices $i,j\in\{2,3,4,5\}$ such that $3r_i, 3r_j\not\in T$. Without loss of generality, let
$3r_2,3r_3\not\in T$, which implies $6r_2, 6r_3\not\in T$.
Moreover, since $\gcd\{r_1,b\}=\gcd\{r_1,6r_2\}=r_2$, $\gcd\{r_1,c\}=\gcd\{r_1,6r_3\}=r_3$, we have $\gcd\{\frac{r_1}{r_2},6\}=1$ and $\gcd\{\frac{r_1}{r_3},6\}=1$.
It is straightforward to verify that
$\gcd\{b,f\}=\gcd\{6r_2,3r_1\}=3r_2$ and $\gcd\{3r_2,a\}=\gcd\{3r_2,2r_1\}=r_2$.
Define $\gcd\{3r_2,c\}=t_1$, $\gcd\{3r_2,d\}=t_2$, $\gcd\{3r_2,e\}=t_3$.
Then one readily checks that
$\frac{b}{3r_2}\frac{f}{3r_2}\frac{a}{r_2}\frac{c}{t_1}\frac{d}{t_2}\frac{e}{t_3}
=2^1\frac{r_1}{r_2}\frac{2r_1}{r_2}\frac{c}{t_1}\frac{d}{t_2}\frac{e}{t_3}$.

When $\frac{r_1}{r_2}=1$, we have $\gcd\{b,c\}=6r_3$, $\gcd\{6r_3,a\}=2r_3$, $\gcd\{6r_3,f\}=3r_3$. Let $\gcd\{6r_3,d\}=t_2$ and $\gcd\{6r_3,3\}=t_3$.
Then $\frac{b}{6r_3}\frac{c}{6r_3}\frac{a}{2r_3}\frac{f}{3r_3}
\frac{d}{t_2}\frac{e}{t_3}=\frac{r_1}{r_3}1^1\frac{r_1}{r_3}\frac{r_1}{r_3}\frac{d}{t_2}\frac{e}{t_3}
\not\in\mathcal{T}$. By Lemmas \ref{k=23}-\ref{k=5} and \ref{split}, there exists an
MOA$(M_1;\frac{b}{6r_3}\frac{c}{6r_3}\frac{a}{2r_3}\frac{f}{3r_3}
\frac{d}{t_2}\frac{e}{t_3},2)$, where $M_1=\operatorname{lcm}\{\alpha\beta:\{\alpha,\beta\}\subseteq
\{\frac{b}{6r_3},\frac{c}{6r_3},\frac{a}{2r_3},\frac{f}{3r_3},
\frac{d}{t_2},\frac{e}{t_3}\}\}$. There exists an MOA$((6r_3)^2;(6r_3)^2(2r_3)^1(3r_3)^1t_2t_3,2)$ by Lemmas \ref{lem3.2.1} and \ref{jj}. Applying Lemmas \ref{jb1} and \ref{bbb} yields an MOA$(M;abcdef,2)$.

Now suppose $\frac{r_1}{r_2}\neq1$. Then $\frac{b}{3r_2}\frac{f}{3r_2}\frac{a}{r_2}\frac{c}{t_1}\frac{d}{t_2}\frac{e}{t_3}=2^1\frac{r_1}{r_2}\frac{2r_1}{r_2}\frac{c}{t_1}\frac{d}{t_2}\frac{e}{t_3}
\not\in\mathcal{T}$ or $2^1\frac{r_1}{r_2}\frac{2r_1}{r_2}\frac{c}{t_1}\frac{d}{t_2}\frac{e}{t_3}=2^15^110^4$.
In the former case, by Lemmas \ref{k=4}-\ref{k=5} and \ref{split}, there exists an MOA$(M_1;\frac{b}{3r_2}\frac{f}{3r_2}\frac{a}{r_2}\frac{c}{t_1}\frac{d}{t_2}\frac{e}{t_3},2)$, where $M_1=\operatorname{lcm}\{\alpha\beta:\{\alpha,\beta\}\subseteq\{\frac{b}{3r_2},\frac{f}{3r_2},\frac{a}{r_2},\frac{c}{t_1},\frac{d}{t_2},\frac{e}{t_3}\}\}$. By Lemmas \ref{lem3.2.1} and \ref{jj}, we obtain an MOA$((3r_2)^2;(3r_2)^2r_2t_1t_2t_3,2)$. Applying Lemmas \ref{jb1} and \ref{bbb} gives an MOA$(M;abcdef,2)$. In the latter case, where  $2^1\frac{r_1}{r_2}\frac{2r_1}{r_2}\frac{c}{t_1}\frac{d}{t_2}\frac{e}{t_3}=2^15^110^4$, we have $2^1\frac{r_1}{r_3}\frac{2r_1}{r_3}\frac{c}{t_1}\frac{d}{t_2}\frac{e}{t_3}\neq2^15^110^4$. Moreover, $\gcd\{c,f\}=\gcd\{6r_3,3r_1\}=3r_3$, $\gcd\{3r_3,a\}=\gcd\{3r_3,2r_1\}=r_3$. Define $\gcd\{3r_3,b\}=t_1$, $\gcd\{3r_3,d\}=t_2$, $\gcd\{3r_3,e\}=t_3$. Then clearly $\frac{c}{3r_3}\frac{f}{3r_3}\frac{a}{r_3}\frac{b}{t_1}\frac{d}{t_2}\frac{e}{t_3}
=2^1\frac{r_1}{r_3}\frac{2r_1}{r_3}\frac{b}{t_1}\frac{d}{t_2}\frac{e}{t_3}\neq2^13^16^4$, otherwise, $r_1=3r_3$, which would imply $a=c$, contradicting the assumption that $a\neq c$. Thus, $\frac{c}{3r_3}\frac{f}{3r_3}\frac{a}{r_3}\frac{b}{t_1}\frac{d}{t_2}\frac{e}{t_3}\not\in
\mathcal{T}$ or $r_1=r_3$. In either case, the same argument as above yields an MOA$(M;abcdef,2)$. This completes the proof.
\end{proof}

\begin{lem}\label{eee}
Let $a,b,c,d,e$ and $f$ be pairwise distinct positive integers greater than $1$. If
$\gcd\{x,y\}\in T=\{2,3,4,6,10,22\}$ for any $\{x,y\}\subseteq \{a,b,c,d,e,f\}$,
then an MOA$(M;abcdef,2)$ exists.
\end{lem}
\begin{proof}
We assume without loss of generality that
\begin{center}
$\gcd\{a,f\}=r_1=max\{\gcd\{x,y\}:\{x,y\}\subseteq\{a,b,c,d,e,f\}\}$,
\end{center}
$\gcd\{r_1,b\}=r_2,\ \gcd\{r_1,c\}=r_3,\ \gcd\{r_1,d\}=r_4,\ \gcd\{r_1,e\}=r_5$.

For $r_1=r_2=r_3=r_4=r_5$, we have $\gcd\{x,y\}=r_1$ for any $\{x,y\}\subseteq\{a,b,c,d,e,f\}$. It is easy to compute that $M=\operatorname{lcm}\{xy:\{x,y\}$$\subseteq$$\{a,b,c,d,e,f\}\}=\frac{abcdef}{r_1^4}$.
Lemma \ref{3.2.5} gives an MOA$(\frac{def}{r_1};r_1^3def,2)$. Applying Lemma \ref{jb} with $n=\frac{a}{r_1}, \frac{b}{r_1}$ and $\frac{c}{r_1}$, respectively,
an MOA$(M;abcdef,2)$ exists. In fact, by assumption, we have $r_i\mid r_1$ and $r_i>1$ for $i=2,3,4,5$.

(1) For $r_1\in\{2,3\}$. We have $r_1=r_2=r_3=r_4=r_5$, it follows that the conclusion holds.

(2) For $r_1=4$. It is clear that none of $r_2,r_3,r_4,r_5$ equals to $1$.
When exactly one of \(r_2,r_3,r_4,r_5\) equals $2$. We assume that $r_2=r_3=r_4=4$, $r_5=2$. Since $\gcd\{r_1,\frac{4e}{r_5}\}=4$, there exists an MOA$(M;bcd(\frac{4e}{r_5})af,2)$, then applying Lemma \ref{jj} yields an MOA$(M;abcdef,2)$.

When $\{r_2,r_3,r_4,r_5\}\in\{\{4,4,2,2\},\{4,2,2,2\},\{2,2,2,2\}\}$.
By Lemmas \ref{3.2.5} and \ref{jj}, there exists an MOA$(r_1^2; r_1^2r_2r_3r_4r_5, 2)$. Since $\gcd\{\frac{a}{r_1},\frac{f}{r_1}\}=1$, and $b,c,d,e$ are pairwise distinct, by Lemmas \ref{k=5} and \ref{split}, there exists an MOA$(M';\frac{a}{r_1}\frac{f}{r_1}\frac{b}{r_2}\frac{c}{r_3}\frac{d}{r_4}\frac{e}{r_5},2)$.
Then applying Lemma \ref{jb1} produces an MOA$(r_1^2M';abcdef,2)$, that is, an MOA$(M;abcdef,2)$ by Lemma \ref{bbb}.

(3) For $r_1=6$. Note that none of $r_2,r_3,r_4,r_5$ equals to $1$.
 When $r_2=6$, $r_3,r_4,r_5\in\{2,3,6\}$ and $r_3,r_4,r_5$ are not all equivalent to $6$, it is easy to see that
$\gcd\{r_1,\frac{6c}{r_3}\}=6,\ \gcd\{r_1,\frac{6d}{r_4}\}=6,\ \gcd\{r_1,\frac{6e}{r_5}\}=6$. Consequently, there exists an MOA$(M;b\frac{6c}{r_3}\frac{6d}{r_4}\frac{6e}{r_5}af,2)$,
where $M=\operatorname{lcm}\{\alpha\beta:\{\alpha,\beta\}\subseteq\{b,\frac{6c}{r_3},\frac{6d}{r_4},\frac{6e}{r_5},a,f\}\}=
\operatorname{lcm}\{\alpha\beta:\{\alpha,\beta\}\subseteq\{\frac{6b}{6},\frac{6c}{r_3},\frac{6d}{r_4},\frac{6e}{r_5},
\frac{6a}{6},\frac{6f}{6}\}\}=r_1^2M'$ by Lemma \ref{bbb}.

When $r_2,r_3,r_4,r_5\in\{2,3\}$. By Lemmas \ref{36100} and \ref{jj}, there exists an MOA$(r_1^2;r_1^2r_2r_3r_4r_5,2)$. Since $\gcd\{\frac{a}{r_1},\frac{f}{r_1}\}=1$ and $b,c,d,e$ are pairwise distinct, there exists an MOA$(M';\frac{a}{r_1}\frac{f}{r_1}\frac{b}{r_2}\frac{c}{r_3}\frac{d}{r_4}\frac{e}{r_5},2)$ by Lemmas \ref{k=5} and \ref{split} (note that $\frac{a}{r_1}\frac{f}{r_1}\frac{b}{r_2}\frac{c}{r_3}\frac{d}{r_4}\frac{e}{r_5}\not\in \mathcal{L}_5$ in Lemma \ref{k=5}, and there are similar situations below, we will not repeat them one by one). It follows that there exists an MOA$(r_1^2M';abcdef,2)$ by Lemma \ref{jb1}, that is, an MOA$(M;abcdef,2)$ by Lemma \ref{bbb}.

(4) For $r_1=10$. When $r_2=r_3=r_4=10$, $r_5\in\{1,2,5\}$. Since $\gcd\{r_1,\frac{10e}{r_5}\}=10$, there exists an MOA$(M;bcd(\frac{10e}{r_5})af,2)$.
By Lemma \ref{jj}, there exists an MOA$(M;abcdef,2)$.

When $r_2=r_3=10,r_4\in\{1,2,5\},r_5\in\{1,2,5\}$. Since $\gcd\{r_1,\frac{10d}{r_4}\}=10$ and $\gcd\{r_1,\frac{10e}{r_5}\}=10$, an MOA$(M;bc(\frac{10d}{r_4})(\frac{10e}{r_5})af,2)$ exists. Applying Lemma \ref{jj} yields an MOA$(M;abcdef,2)$.

When $r_2=10,\{r_3,r_4,r_5\}\in\{\{2,2,2\},\{5,2,2\},\{5,5,2\},\{5,5,5\}\}$. By Lemma \ref{3.2.5}, there exists an MOA$(r_1^2;r_1^2r_2r_3r_4r_5,2)$.
Since $\gcd\{\frac{a}{r_1},\frac{f}{r_1}\}=1$, and $b,c,d,e$ are pairwise distinct,
there exists an  MOA$(M';\frac{a}{r_1}\frac{f}{r_1}\frac{b}{r_2}\frac{c}{r_3}\frac{d}{r_4}\frac{e}{r_5},2)$
by Lemmas \ref{k=5} and \ref{split}. Then there exists an MOA$(M;abcdef,2)$ by Lemmas \ref{jb1} and \ref{bbb}.

When $r_2,r_3,r_4,r_5\in\{2,5\}$. By Lemma \ref{36100}, there exists an MOA$(r_1^2;r_1^2r_2r_3r_4r_5,2)$. Since $\gcd\{\frac{a}{r_1},\frac{f}{r_1}\}=1$, and $b,c,d$ and $e$ are pairwise distinct, there exists an MOA$(M';(\frac{a}{r_1}\frac{f}{r_1})\frac{b}{r_2}\frac{c}{r_3}\frac{d}{r_4}\frac{e}{r_5},2)$
by Lemma \ref{k=5}. Then there exists an MOA$(r_1^2M';abcdef,2)$ by Lemmas \ref{jb1} and \ref{bbb}.

(5) For $r_1=22$. When $r_2=r_3=r_4=22$, $r_5\in\{2,11\}$. Since $\gcd\{r_1,\frac{22e}{r_5}\}=22$, there exists an MOA$(M;bcd(\frac{22e}{r_5})af,2)$. Applying Lemma \ref{jj} produces an MOA$(M;abcdef,2)$.

When $r_2=r_3=22$, $r_4, r_5\in\{2,11\}$. Since $\gcd\{r_1,\frac{22d}{r_4}\}=22$ and $\gcd\{r_1,\frac{22e}{r_5}\}=22$,
there exists an MOA$(M;bc(\frac{22d}{r_4})(\frac{22e}{r_5})af,2)$. Applying
Lemma \ref{jj} yields an MOA$(M;abcdef,2)$.

When $r_2=22$, $r_i\in\{2,11\}$ for $i=3,4,5$. There exists an MOA$(r_1^2;r_1^2r_2r_3r_4r_5,2)$ by Lemma \ref{3.2.5}.
Since $\gcd\{\frac{a}{r_1},\frac{f}{r_1}\}=1$, and $b,c,d$ and $e$ are pairwise distinct, then there exists an MOA$(M';\frac{a}{r_1}\frac{f}{r_1}\frac{b}{r_2}\frac{c}{r_3}\frac{d}{r_4}\frac{e}{r_5},2)$.
Thus, there exists an MOA$(r_1^2M';abcdef,2)$ by Lemma \ref{jb1}.

When $r_i\in\{2,11\},i=2,3,4,5$.
There exists an MOA$(r_1^2;r_1^2r_2r_3r_4r_5,2)$ by Lemmas \ref{36100}, \ref{22-30} and \ref{jj}. Since $\gcd\{\frac{a}{r_1},\frac{f}{r_1}\}=1$, and $b,c,d,e$ are pairwise distinct, then there exists an MOA$(M';\frac{a}{r_1}\frac{f}{r_1}\frac{b}{r_2}\frac{c}{r_3}\frac{d}{r_4}\frac{e}{r_5},2)$.
Thus, applying Lemmas \ref{jb1} and \ref{bbb} yields an MOA$(r_1^2M';abcdef,2)$, that is, an MOA$(M;abcdef,2)$. The proof is completed.
\end{proof}

By Lemma \ref{ddd} and Lemma \ref{eee} we have the following theorem.

\begin{thm}\label{hbxt}
Let $a,b,c,d,e$ and $f$ be pairwise distinct positive integers greater than $1$.
If $\gcd\{x,y\}>1$ for any $\{x,y\}\subseteq \{a,b,c,d,e,f\}$, then there exists an MOA$(M;abcdef,2)$, where $M=\operatorname{lcm}\{xy: \{x,y\}\subset
\{a,b,c,d,e,f\}\}$.
\end{thm}

\vskip 8pt

\textbf{The proof of Theorem \ref{mainR}}\ \ \
By Lemma \ref{jb}, Lemmas \ref{lem3.1}-\ref{3.2.5},
Theorem \ref{6cc} and Theorem \ref{hbxt}, the proof is completed.

\section{Concluding remarks}

In this paper, we investigate mixed orthogonal arrays of strength two with six factors and establish our main result, Theorem \ref{mainR}. However, some problems remain for further investigation. Let $a, b, c, d, e, f$ be integers greater than $1$,
\begin{center}
$M=\operatorname{lcm}\{\alpha\beta: \{\alpha,\beta\}\subset \{a,b,c,d,e,f\}\}$,
\end{center}
\begin{center}
$\mathcal{P}=\{(2s)^6, (2s)^5s^1, (2s)^52^1, (2s)^4s^2: s\in\{5,11\}\}\cup\{10^42^15^1\}$.
\end{center}
By Theorem \ref{mainR}, an MOA$(uM;abcdef,2)$ exists for $abcdef\in \mathcal{P}$, where $u>1$. Nine MOAs remain undetermined, and the existence of these nine MOAs forms an interesting open problem worthy of further investigation. Meanwhile, Zhu et al. \cite{Zhu-Pang-2026} obtained several isolated new results on mixed orthogonal arrays of strength two with seven factors. The existence theory for mixed orthogonal arrays of strength two with seven factors is far from complete and merits further study. We state these open problems below.

\textbf{Problem 1} \ Does there exist an MOA$(M;abcdef,2)$ for $abcdef\in \mathcal{P}$?

\textbf{Problem 2} \ Let $a,b,c,d,e,f,g$ be integers greater than $1$, and define
\begin{center}
$M^*=\operatorname{lcm}\{\alpha\beta: \{\alpha,\beta\}\subset \{a,b,c,d,e,f,g\}\}$.
\end{center}
Determine whether an MOA$(M^*;abcdefg,2)$ exists.

\vskip 12pt
\noindent {\bf Acknowledgements}

This work was supported by Science Research Project of Hebei Education Department (Grant No. JCZX2026032), Natural Science Foundation of Hebei Province (Grant Nos. A2025 205023 and A2023205045), National Natural Science Foundation of China (Grant No. 62372157) and Natural Science Foundation of Henan Province (Grant No. 262300421846).


\newpage
\section*{Appendix I}

\begin{center}
\textbf{The Construction of an RTD$_3(5,6)$}
\end{center}

\noindent \textbf{Point Set:} $\mathcal{X} = \mathbb{Z}_5\times \mathbb{Z}_2\times \mathbb{Z}_3$;

\noindent \textbf{Group Set:} $\mathcal{G} = \{\{i\}\times \mathbb{Z}_2\times \mathbb{Z}_3: i\in \mathbb{Z}_5\}$;

\noindent \textbf{The required block set:} $\mathcal{D}=\{d_1,d_2,\ldots,d_{108}\}$, where we simply write the point $(x,y,z)\in \mathcal{X}$  as $xyz$.

\noindent \textbf{Parallel classes:} $\mathcal{Q}_{i}(i=1,2,\ldots,18)$.

\vskip8pt
\noindent
$d_{1}$$=$$\{000,100,200,300,400\}$,~~~$d_{2}$$=$$\{000,101,201,301,400\}$,~~~$d_{3}$$=$$\{000,102,202,302,400\}$,\\
$d_{4}$$=$$\{001,112,211,302,400\}$,~~~$d_{5}$$=$$\{001,110,212,300,400\}$,~~~$d_{6}$$=$$\{001,111,210,301,400\}$,\\
$d_{7}$$=$$\{010,102,200,311,400\}$,~~~$d_{8}$$=$$\{010,100,201,312,400\}$,~~~$d_{9}$$=$$\{010,101,202,310,400\}$,\\
$d_{10}$$=$$\{011,100,211,302,400\}$,~~~$d_{11}$$=$$\{011,101,212,300,400\}$,~~~$d_{12}$$=$$\{011,102,210,301,400\}$,\\
$d_{13}$$=$$\{012,110,212,311,400\}$,~~~$d_{14}$$=$$\{012,111,210,312,400\}$,~~~$d_{15}$$=$$\{012,112,211,310,400\}$,\\
$d_{16}$$=$$\{002,110,202,311,400\}$,~~~$d_{17}$$=$$\{002,111,200,312,400\}$,~~~$d_{18}$$=$$\{002,112,201,310,400\}$,\\
$d_{19}$$=$$\{001,111,211,311,401\}$,~~~$d_{20}$$=$$\{001,112,212,312,401\}$,~~~$d_{21}$$=$$\{001,110,210,310,401\}$,\\
$d_{22}$$=$$\{002,101,202,311,401\}$,~~~$d_{23}$$=$$\{002,102,200,312,401\}$,~~~$d_{24}$$=$$\{002,100,201,310,401\}$,\\
$d_{25}$$=$$\{000,101,210,301,401\}$,~~~$d_{26}$$=$$\{000,102,211,302,401\}$,~~~$d_{27}$$=$$\{000,100,212,300,401\}$,\\
$d_{28}$$=$$\{012,111,202,300,401\}$,~~~$d_{29}$$=$$\{012,112,200,301,401\}$,~~~$d_{30}$$=$$\{012,110,201,302,401\}$,\\
$d_{31}$$=$$\{010,112,200,302,401\}$,~~~$d_{32}$$=$$\{010,110,201,300,401\}$,~~~$d_{33}$$=$$\{010,111,202,301,401\}$,\\
$d_{34}$$=$$\{011,102,211,310,401\}$,~~~$d_{35}$$=$$\{011,100,212,311,401\}$,~~~$d_{36}$$=$$\{011,101,210,312,401\}$,\\
$d_{37}$$=$$\{002,101,212,302,402\}$,~~~$d_{38}$$=$$\{002,102,210,300,402\}$,~~~$d_{39}$$=$$\{002,100,211,301,402\}$,\\
$d_{40}$$=$$\{011,111,201,301,402\}$,~~~$d_{41}$$=$$\{011,112,202,302,402\}$,~~~$d_{42}$$=$$\{011,110,200,300,402\}$,\\
$d_{43}$$=$$\{012,112,202,300,402\}$,~~~$d_{44}$$=$$\{012,110,200,301,402\}$,~~~$d_{45}$$=$$\{012,111,201,302,402\}$,\\
$d_{46}$$=$$\{010,102,210,311,402\}$,~~~$d_{47}$$=$$\{010,100,211,312,402\}$,~~~$d_{48}$$=$$\{010,101,212,310,402\}$,\\
$d_{49}$$=$$\{000,111,210,312,402\}$,~~~$d_{50}$$=$$\{000,112,211,310,402\}$,~~~$d_{51}$$=$$\{000,110,212,311,402\}$,\\
$d_{52}$$=$$\{001,101,201,312,402\}$,~~~$d_{53}$$=$$\{001,102,202,310,402\}$,~~~$d_{54}$$=$$\{001,100,200,311,402\}$,\\
$d_{55}$$=$$\{010,110,210,310,410\}$,~~~$d_{56}$$=$$\{010,111,211,311,410\}$,~~~$d_{57}$$=$$\{010,112,212,312,410\}$,\\
$d_{58}$$=$$\{000,110,200,310,410\}$,~~~$d_{59}$$=$$\{000,111,201,311,410\}$,~~~$d_{60}$$=$$\{000,112,202,312,410\}$,\\
$d_{61}$$=$$\{001,100,201,302,410\}$,~~~$d_{62}$$=$$\{001,101,202,300,410\}$,~~~$d_{63}$$=$$\{001,102,200,301,410\}$,\\
$d_{64}$$=$$\{002,112,212,301,410\}$,~~~$d_{65}$$=$$\{002,110,210,302,410\}$,~~~$d_{66}$$=$$\{002,111,211,300,410\}$,\\
$d_{67}$$=$$\{011,102,201,310,410\}$,~~~$d_{68}$$=$$\{011,100,202,311,410\}$,~~~$d_{69}$$=$$\{011,101,200,312,410\}$,\\
$d_{70}$$=$$\{012,100,212,301,410\}$,~~~$d_{71}$$=$$\{012,101,210,302,410\}$,~~~$d_{72}$$=$$\{012,102,211,300,410\}$,\\
$d_{73}$$=$$\{011,112,201,301,411\}$,~~~$d_{74}$$=$$\{011,110,202,302,411\}$,~~~$d_{75}$$=$$\{011,111,200,300,411\}$,\\
$d_{76}$$=$$\{012,102,212,312,411\}$,~~~$d_{77}$$=$$\{012,100,210,310,411\}$,~~~$d_{78}$$=$$\{012,101,211,311,411\}$,\\
$d_{79}$$=$$\{002,111,212,310,411\}$,~~~$d_{80}$$=$$\{002,112,210,311,411\}$,~~~$d_{81}$$=$$\{002,110,211,312,411\}$,\\
$d_{82}$$=$$\{000,101,200,310,411\}$,~~~$d_{83}$$=$$\{000,102,201,311,411\}$,~~~$d_{84}$$=$$\{000,100,202,312,411\}$,\\
$d_{85}$$=$$\{001,101,211,300,411\}$,~~~$d_{86}$$=$$\{001,102,212,301,411\}$,~~~$d_{87}$$=$$\{001,100,210,302,411\}$,\\
$d_{88}$$=$$\{010,111,200,302,411\}$,~~~$d_{89}$$=$$\{010,112,201,300,411\}$,~~~$d_{90}$$=$$\{010,110,202,301,411\}$,\\
$d_{91}$$=$$\{012,102,202,312,412\}$,~~~$d_{92}$$=$$\{012,100,200,310,412\}$,~~~$d_{93}$$=$$\{012,101,201,311,412\}$,\\
$d_{94}$$=$$\{010,100,210,300,412\}$,~~~$d_{95}$$=$$\{010,101,211,301,412\}$,~~~$d_{96}$$=$$\{010,102,212,302,412\}$,\\
$d_{97}$$=$$\{011,110,211,312,412\}$,~~~$d_{98}$$=$$\{011,111,212,310,412\}$,~~~$d_{99}$$=$$\{011,112,210,311,412\}$,\\
$d_{100}$$=$$\{001,110,201,312,412\}$,~~~$d_{101}$$=$$\{001,111,202,310,412\}$,~~~$d_{102}$$=$$\{001,112,200,311,412\}$,\\
$d_{103}$$=$$\{002,100,202,301,412\}$,~~~$d_{104}$$=$$\{002,101,200,302,412\}$,~~~$d_{105}$$=$$\{002,102,201,300,412\}$,\\
$d_{106}$$=$$\{000,112,210,300,412\}$,~~~$d_{107}$$=$$\{000,110,211,301,412\}$,~~~$d_{108}$$=$$\{000,111,212,302,412\}$.\\[.2cm]
~$\mathcal{Q}_{1}$$=$$\{d_{1},d_{55},d_{19},d_{73},d_{37},d_{91}\}$,
~$\mathcal{Q}_{2}$$=$$\{d_{2},d_{56},d_{20},d_{74},d_{38},d_{92}\}$,
~$\mathcal{Q}_{3}$$=$$\{d_{3},d_{57},d_{21},d_{75},d_{39},d_{93}\}$,\\
~$\mathcal{Q}_{4}$$=$$\{d_{58},d_{94},d_{76},d_{22},d_{4},d_{40}\}$,
~$\mathcal{Q}_{5}$$=$$\{d_{59},d_{95},d_{77},d_{23},d_{5},d_{41}\}$,
~$\mathcal{Q}_{6}$$=$$\{d_{60},d_{96},d_{78},d_{24},d_{6},d_{42}\}$,\\
~$\mathcal{Q}_{7}$$=$$\{d_{61},d_{97},d_{79},d_{43},d_{25},d_{7}\}$,
~$\mathcal{Q}_{8}$$=$$\{d_{62},d_{98},d_{80},d_{44},d_{26},d_{8}\}$,
~$\mathcal{Q}_{9}$$=$$\{d_{63},d_{99},d_{81},d_{45},d_{27},d_{9}\}$,\\
~$\mathcal{Q}_{10}$$=$$\{d_{100},d_{10},d_{46},d_{82},d_{64},d_{28}\}$,
~$\mathcal{Q}_{11}$$=$$\{d_{101},d_{11},d_{47},d_{83},d_{65},d_{29}\}$,
~$\mathcal{Q}_{12}$$=$$\{d_{102},d_{12},d_{48},d_{84},d_{66},d_{30}\}$,\\
~$\mathcal{Q}_{13}$$=$$\{d_{103},d_{13},d_{49},d_{31},d_{85},d_{67}\}$,
~$\mathcal{Q}_{14}$$=$$\{d_{104},d_{14},d_{50},d_{32},d_{86},d_{68}\}$,
~$\mathcal{Q}_{15}$$=$$\{d_{105},d_{15},d_{51},d_{33},d_{87},d_{69}\}$,\\
~$\mathcal{Q}_{16}$$=$$\{d_{16},d_{70},d_{34},d_{52},d_{106},d_{88}\}$,
~$\mathcal{Q}_{17}$$=$$\{d_{17},d_{71},d_{35},d_{53},d_{107},d_{89}\}$,
~$\mathcal{Q}_{18}$$=$$\{d_{18},d_{72},d_{36},d_{54},d_{108},d_{90}\}$.

\newpage
\section*{Appendix II}

\footnotesize{}
$$\begin{aligned}
P&=\gcd\{bcde,\ bcda,\ bcdf,\ bcea,\ bcef,\ bcaf,\ bdea,\ bdef,\ bdaf,\ beaf,\ cdea,\ cdef,\ cdaf,\ ceaf,\ deaf\}\\
&=\gcd\{bcde,\ \gcd\{bcda,bcdf\},\ \gcd\{bcea,bcef\}, \ bcaf,\ \gcd\{bdea,bdef\},\ bdaf,\ beaf,\ \gcd\{cdea,\\
& \  \  \  \  cdef\},\ cdaf,\ ceaf,\ deaf\}\\
&=\gcd\{bcde,\ bcdr_1,\ bcer_1,\ bder_1,\ cder_1,\ bcaf,\ bdaf,\ beaf,\ cdaf,\ ceaf,\ deaf\}\\
&=\gcd\{\gcd\{bcde, bcdr_1\},\ g,c,d,\{bcde, bcer_1\},\ \gcd\{bcde, bder_1\},\ \gcd\{bcde, cder_1\},\ bcaf,\ bdaf,\\
& \ \ \ \ beaf,\ cdaf,\ ceaf,\ deaf\}\\
&=\gcd\{bcdr_5,\ bcer_4,\ bder_3,\ cder_2,\ bcaf,\ bdaf,\ beaf,\ cdaf,\ ceaf,\ deaf\}\\
&=r_2r_3r_4r_5 \times \gcd\{\frac{b}{r_2}\frac{c}{r_3}\frac{a}{r_4}\frac{f}{r_5},\ \frac{b}{r_2}\frac{d}{r_4}\frac{a}{r_3}\frac{f}{r_5},
\ \frac{b}{r_2}\frac{e}{r_5}\frac{a}{r_3}\frac{f}{r_4},\ \frac{c}{r_3}\frac{d}{r_4}\frac{a}{r_2}\frac{f}{r_5},\
\frac{c}{r_3}\frac{e}{r_5}\frac{a}{r_2}\frac{f}{r_4},\ \frac{d}{r_4}\frac{e}{r_5}\frac{a}{r_2}\frac{f}{r_3},\
\frac{b}{r_2}\frac{c}{r_3}\frac{d}{r_4},\\
& \ \ \ \ \frac{b}{r_2}\frac{c}{r_3}\frac{e}{r_5},\ \frac{b}{r_2}\frac{d}{r_4}\frac{e}{r_5},\
\frac{c}{r_3}\frac{d}{r_4}\frac{e}{r_5}\}\\
&=r_2r_3r_4r_5 \times \gcd\{\frac{b}{r_2}\frac{c}{r_3}\frac{a}{r_1}\frac{f}{r_1}\frac{r_1}{r_4}\frac{r_1}{r_5},
\frac{b}{r_2}\frac{d}{r_4}\frac{a}{r_1}\frac{f}{r_1}\frac{r_1}{r_3}\frac{r_1}{r_5},\
\frac{b}{r_2}\frac{e}{r_5}\frac{a}{r_1}\frac{f}{r_1}\frac{r_1}{r_3}\frac{r_1}{r_4},\
\frac{c}{r_3}\frac{d}{r_4}\frac{a}{r_1}\frac{f}{r_1}\frac{r_1}{r_2}\frac{r_1}{r_5},\
\frac{c}{r_3}\frac{e}{r_5}\frac{a}{r_1}\frac{f}{r_1}\frac{r_1}{r_2}\frac{r_1}{r_4},\\
&\ \  \  \
\frac{d}{r_4}\frac{e}{r_5}\frac{a}{r_1}\frac{f}{r_1}\frac{r_1}{r_2}\frac{r_1}{r_3},\
\frac{b}{r_2}\frac{c}{r_3}\frac{d}{r_4},\frac{b}{r_2}\frac{c}{r_3}\frac{e}{r_5},\
\frac{b}{r_2}\frac{d}{r_4}\frac{e}{r_5},\frac{c}{r_3}\frac{d}{r_4}\frac{e}{r_5} \}\\
&=r_2r_3r_4r_5 \times \gcd\{\gcd\{\frac{b}{r_2}\frac{c}{r_3}\frac{a}{r_1}\frac{f}{r_1}\frac{r_1}{r_4}\frac{r_1}{r_5},\
\frac{b}{r_2}\frac{c}{r_3}\frac{d}{r_4},\frac{b}{r_2}\frac{c}{r_3}\frac{e}{r_5}\},\
\gcd\{\frac{b}{r_2}\frac{d}{r_4}\frac{a}{r_1}\frac{f}{r_1}\frac{r_1}{r_3}\frac{r_1}{r_5},\ \frac{b}{r_2}\frac{c}{r_3}\frac{d}{r_4},\ \frac{b}{r_2}\frac{d}{r_4}\frac{e}{r_5}\},
\\
&\ \  \ \
\gcd\{\frac{b}{r_2}\frac{e}{r_5}\frac{a}{r_1}\frac{f}{r_1}\frac{r_1}{r_3}\frac{r_1}{r_4},\
\frac{b}{r_2}\frac{c}{r_3}\frac{e}{r_5},\ \frac{b}{r_2}\frac{d}{r_4}\frac{e}{r_5}\},\
\gcd\{\frac{c}{r_3}\frac{d}{r_4}\frac{a}{r_1}\frac{f}{r_1}\frac{r_1}{r_2}\frac{r_1}{r_5},\
\frac{b}{r_2}\frac{c}{r_3}\frac{d}{r_4},\ \frac{c}{r_3}\frac{d}{r_4}\frac{e}{r_5}\},\\
&\ \  \  \  \gcd\{\frac{c}{r_3}\frac{e}{r_5}\frac{a}{r_1}\frac{f}{r_1}\frac{r_1}{r_2}\frac{r_1}{r_4},\
\frac{b}{r_2}\frac{c}{r_3}\frac{e}{r_5},\ \frac{c}{r_3}\frac{d}{r_4}\frac{e}{r_5} \},\
\gcd\{\frac{d}{r_4}\frac{e}{r_5}\frac{a}{r_1}\frac{f}{r_1}\frac{r_1}{r_2}\frac{r_1}{r_3},\
\frac{b}{r_2}\frac{d}{r_4}\frac{e}{r_5},\ \frac{c}{r_3}\frac{d}{r_4}\frac{e}{r_5}\}\}
\\
&=r_2r_3r_4r_5 \times \gcd\{\frac{b}{r_2}\frac{c}{r_3}\frac{a}{r_1}\frac{f}{r_1},\
\frac{b}{r_2}\frac{d}{r_4}\frac{a}{r_1}\frac{f}{r_1},\
\frac{b}{r_2}\frac{e}{r_5}\frac{a}{r_1}\frac{f}{r_1},\
\frac{c}{r_3}\frac{d}{r_4}\frac{a}{r_1}\frac{f}{r_1},\
\frac{c}{r_3}\frac{e}{r_5}\frac{a}{r_1}\frac{f}{r_1},\
\frac{d}{r_4}\frac{e}{r_5}\frac{a}{r_1}\frac{f}{r_1},\\
& \ \ \ \  \frac{b}{r_2}\frac{c}{r_3}\frac{d}{r_4},\frac{b}{r_2}\frac{c}{r_3}\frac{e}{r_5},\
\frac{b}{r_2}\frac{d}{r_4}\frac{e}{r_5},\
\frac{c}{r_3}\frac{d}{r_4}\frac{e}{r_5}\}.
\end{aligned}$$

$$\begin{aligned}
P'&=\gcd\{\frac{bcde}{r_2r_3r_4r_5},\frac{bcda}{r_2r_3r_4r_1},\
\frac{bcdf}{r_2r_3r_4r_1},\frac{bcea}{r_2r_3r_5r_1},\ \frac{bcef}{r_2r_3r_5r_1},\
\frac{bcaf}{r_2r_3r_1r_1},\ \frac{bdea}{r_2r_4r_5r_1},\ \frac{bdef}{r_2r_4r_5r_1},\\[.2cm]
& \ \ \ \ \ \ \frac{bdaf}{r_2r_4r_1r_1},\ \frac{beaf}{r_2r_5r_1r_1},\
\frac{cdea}{r_3r_4r_5r_1},\ \frac{cdef}{r_3r_4r_5r_1},\ \frac{cdaf}{r_3r_4r_1r_1},\
\frac{ceaf}{r_3r_5r_1r_1},\ \frac{deaf}{r_4r_5r_1r_1}\}\\[.2cm]
&=\gcd\{\frac{bcde}{r_2r_3r_4r_5},\ \gcd\{\frac{bcda}{r_2r_3r_4r_1},\
\frac{bcdf}{r_2r_3r_4r_1}\},\ \gcd\{\frac{bcea}{r_2r_3r_5r_1},\ \frac{bcef}{r_2r_3r_5r_1}\},\
\gcd\{\frac{bdea}{r_2r_4r_5r_1},\ \frac{bdef}{r_2r_4r_5r_1}\},\\[.2cm]
& \ \ \ \ \ \ \frac{bcaf}{r_2r_3r_1r_1},\ \frac{bdaf}{r_2r_4r_1r_1},\ \frac{beaf}{r_2r_5r_1r_1},\
\gcd\{\frac{cdea}{r_3r_4r_5r_1},\ \frac{cdef}{r_3r_4r_5r_1}\},\ \frac{cdaf}{r_3r_4r_1r_1},\
\frac{ceaf}{r_3r_5r_1r_1},\ \frac{deaf}{r_4r_5r_1r_1}\}\\[.2cm]
&=\gcd\{\frac{bcde}{r_2r_3r_4r_5},\ \frac{bcd}{r_2r_3r_4},\ \frac{bce}{r_2r_3r_5},\
\frac{bde}{r_2r_4r_5},\ \frac{bcaf}{r_2r_3r_1r_1},\ \frac{bdaf}{r_2r_4r_1r_1},\frac{beaf}{r_2r_5r_1r_1},
\ \frac{cde}{r_3r_4r_5},\frac{cdaf}{r_3r_4r_1r_1},\\
& \ \ \ \ \
\frac{ceaf}{r_3r_5r_1r_1},\ \frac{deaf}{r_4r_5r_1r_1}\}\\[.2cm]
&=\gcd\{\frac{bcd}{r_2r_3r_4},\ \frac{bce}{r_2r_3r_5},\ \frac{bde}{r_2r_4r_5},\
\frac{cde}{r_3r_4r_5},\ \frac{bcaf}{r_2r_3r_1r_1},\ \frac{bdaf}{r_2r_4r_1r_1},\
\frac{beaf}{r_2r_5r_1r_1},\ \frac{cdaf}{r_3r_4r_1r_1},\
\frac{ceaf}{r_3r_5r_1r_1},\ \frac{deaf}{r_4r_5r_1r_1}\}.
\end{aligned}$$
Thus, $P=r_2r_3r_4r_5P'$.

\end{document}